\documentclass{amsart}
\usepackage{amsmath}
\usepackage[psamsfonts]{amssymb}
\usepackage[all]{xy}
\usepackage{graphicx}
\usepackage[T1]{fontenc}
\usepackage[bitstream-charter]{mathdesign}
\usepackage{hyperref}
\newtheorem{theorem}{Theorem}[section]
\newtheorem{prop}[theorem]{Proposition}

\newtheorem{cor}[theorem]{Corollary}
\newtheorem{assm}[theorem]{Assumption}

\theoremstyle{remark}
\newtheorem{dfn}[theorem]{Definition}
\newtheorem{remark}[theorem]{Remark}

\def\co{\colon\thinspace}
\def\delbar{\bar{\partial}}

\def\ep{\epsilon}

\def\R{\mathbb{R}}
\def\Z{\mathbb{Z}}

\title{Hofer geometry and cotangent fibers}
\author{Michael Usher}
\address{Department of Mathematics\\University of Georgia\\Athens, GA 30602}
\email{usher@math.uga.edu}
\begin{document}

\begin{abstract}
For a class of Riemannian manifolds that include products of arbitrary compact manifolds with manifolds of nonpositive sectional curvature on the one hand, or with certain positive-curvature examples such as spheres of dimension at least $3$ and compact semisimple Lie groups on the other, we show that the Hamiltonian diffeomorphism group of the cotangent bundle contains as subgroups infinite-dimensional normed vector spaces that are bi-Lipschitz embedded with respect to Hofer's metric; moreover these subgroups can be taken to consist of diffeomorphisms supported in an arbitrary neighborhood of the zero section.  In fact, the orbit of a fiber of the cotangent bundle with respect to any of these subgroups is quasi-isometrically embedded with respect to the induced Hofer metric on the orbit of the fiber under the whole group.  The diffeomorphisms in these subgroups are obtained from reparametrizations of the geodesic flow.  Our proofs involve a study of the Hamiltonian-perturbed Floer complex of a pair of cotangent fibers (or, more generally, of a conormal bundle together with a cotangent fiber).  Although the homology of this complex  vanishes, an analysis of its boundary depth yields the lower bounds on the Lagrangian Hofer metric required for our main results.
\end{abstract}

\maketitle

\section{Introduction}

For a symplectic manifold $(P,\omega)$, let $Ham(P,\omega)$ denote the group of Hamiltonian diffeomorphisms $\phi$ of $P$ which may be obtained as time-one maps of compactly-supported smooth functions $H\co [0,1]\times P\to \R$.  (Thus where $X_{H}(t,\cdot)$ is the time-dependent vector field given by $\omega(\cdot,X_{H}(t,\cdot))=d(H(t,\cdot))$ and where $\{\phi_{H}^{t}\}_{0\leq t\leq 1}$ is defined by $\phi_{H}^{0}=1_P$ and $\frac{d\phi_{H}^{t}}{dt}=X_{H(t,\cdot)}\circ\phi_{H}^{t}$ we have $\phi_{H}^{1}=\phi$).    According to \cite{Ho},\cite{LM}, the ``Hofer norm'' \[ \|\phi\|=\inf\left\{\left.\int_{0}^{1}\left(\max_P H(t,\cdot)-\min_P H(t,\cdot)\right)dt\right|\phi_{H}^{1}=\phi\right\} \] gives rise via the formula $d(\phi,\psi)=\|\phi\circ\psi^{-1}\|$ to a bi-invariant metric $d$ on $Ham(P,\omega)$.  

Despite substantial progress, much remains unknown about the large-scale properties of this metric.  In particular, it is still unknown even whether $d$ is unbounded for every symplectic manifold, though unboundedness has now been established in many cases (see, \emph{e.g.}, \cite{M09},\cite{U11b} and references therein).

In the present note we concentrate on the case where the sympelctic manifold $(P,\omega)$ is the cotangent bundle $(T^*N,d\hat{\theta})$ of a compact smooth manifold $N$, with its standard symplectic structure.  In this case the unboundedness of the Hofer metric has been known at least since  \cite{Mili}, but many more refined questions remain.  We investigate some of these here by considering the geodesic flow on $N$ with respect to a suitable Riemannian metric; in particular, under a Morse-theoretic assumption on the geodesics in $N$, we will exhibit infinite-rank subgroups of $Ham(T^*N,d\hat{\theta})$ which are ``large'' with respect to Hofer's metric.

To set up terminology for one of our main results, if $(N,g)$ is a compact Riemannian manifold, $Q\subset N$ is a compact submanifold, and $x_1\in N\setminus Q$, a \emph{geodesic from $Q$ to $x_1$} is by definition a solution $\eta\co [0,1]\to N$ to the geodesic equation $\nabla_{\eta'}\eta'=0$ subject to the boundary conditions $\eta(0)\in Q$, $\eta'(0)\perp T_{\eta(0)}Q$, and $\eta(1)=x_1$.  Letting $\mathcal{P}_N(Q,x_1)$ denote the space of all piecewise-smooth paths $\eta\co [0,1]\to N$ such that $\eta(0)\in Q$ and $\eta(1)=x_1$, the geodesics from $Q$ to $x_1$ are precisely the critical points of the energy functional $E\co\mathcal{P}_N(Q,x_1)\to\R$ defined by $E(\eta)=\int_{0}^{1}|\eta'(t)|^2dt$.  In particular, any geodesic from $Q$ to $x_1$ has a well-defined Morse index, which will be referred to in Theorem \ref{mainham} below.

Also, let $\R^{\infty}$ denote the direct sum of a collection of copies of $\R$ indexed by $\Z_+$, and for $\vec{a}=\{a_k\}_{k=1}^{\infty}\in \R^{\infty}$ define $osc(\vec{a})=\max_{i,j}|a_i-a_j|$ and $\|\vec{a}\|_{\infty}=\max_i|a_i|$ (these maxima are well-defined since, by definition, any $\vec{a}\in \R^{\infty}$ as all but finitely many $a_i=0$). Obviously one has $\|\vec{a}\|_{\infty}\leq osc(\vec{a})\leq 2\|\vec{a}\|_{\infty}$. We prove:

\begin{theorem}\label{mainham} Let $(N,g)$ be a compact connected Riemannian manifold and suppose that there is a compact submanifold $Q\subset N$, a point $x_1\in N\setminus Q$ which is not a focal point of $Q$, and a homotopy class $\frak{c}\in \pi_0(\mathcal{P}_N(Q,x_1))$ such that: \begin{itemize} \item No geodesics from $Q$ to $x_1$ representing the class $\frak{c}$ have Morse index one. \item Only finitely many geodesics from $Q$ to $x_1$ representing the class $\frak{c}$ have Morse index in $\{0,2\}$.  \item $\dim Q\neq \dim N-2$.  \end{itemize}
 Then for any neighborhood $U$ of the zero section $0_N\subset T^*N$ there is a homomorphism $\Phi\co \R^{\infty}\to  Ham(T^*N,d\hat{\theta})$ such that every diffeomorphism $\Phi(\vec{a})$ has support contained in $U\setminus 0_N$ and such that, for all $\vec{a},\vec{b}\in \R^{\infty}$ \begin{equation}\label{maineq} \|\vec{a}-\vec{b}\|_{\infty}\leq d(\Phi(\vec{a}),\Phi(\vec{b}))\leq osc(\vec{a}-\vec{b}) \end{equation}
\end{theorem}

Thus for all manifolds $N$ admitting metrics $g$, submanifolds $Q$, and points $x_1$ as in Theorem \ref{mainham}, the Hamiltonian diffeomorphism group of an arbitrarily small deleted neighborhood $U\setminus 0_N$ of the zero section contains as a subgroup an infinite-dimensional normed vector space which is bi-Lipschitz embedded in $Ham(T^*N,d\hat{\theta})$ with respect to the Hofer metric.  The hypotheses of Theorem \ref{mainham} can at least formally be weakened somewhat (see Assumption \ref{assm} below; the hypotheses of Theorem \ref{mainham} amount to Assumption \ref{assm} being satisfied with $k=0$), though I do not know any examples of manifolds that satisfy Assumption \ref{assm} for some nonzero $k$ but do not also satisfy it for $k=0$.  Let us point out here some  examples of classes of manifolds to which Theorem \ref{mainham} applies; see the discussion following Assumption \ref{assm} for proofs and additional remarks:
\begin{itemize} \item Any compact Riemannian manifold $(N,g)$ of nonpositive sectional curvature (see Proposition \ref{nonpos}).  Indeed, if $\dim N\neq 2$ we can take $Q$ to consist of an arbitrary single point distinct from $x_1$, while if $\dim N=2$ we can take $Q$ to be a closed geodesic not containing $x_1$.
\item Many positively-curved symmetric spaces, including spheres of dimension at least $3$, all compact semisimple Lie groups, and quaternionic Grassmannians (see Proposition \ref{groupsphere} and Remark \ref{curvop}). 
\item All products $N\times N'$, where $N$ is a compact smooth manifold admitting a Riemannian metric $g$, submanifold $Q$, and point $x_1$ as in the hypothesis of Theorem \ref{mainham}, and where $N'$ is \emph{any}  compact smooth manifold (see Remark \ref{product}).

\end{itemize}


In the course of proving Theorem \ref{mainham}, we establish a result concerning the behavior of the Hofer metric with respect to Lagrangian submanifolds of $T^*N$ that is also of interest.  In general, if $(P,\omega)$ is a symplectic manifold and if $S\subset P$ is a closed subset, let \[ \mathcal{L}(S)=\{\phi(S)|\phi\in Ham(P,\omega)\} \] denote the orbit of $S$ under the Hamiltonian diffeomorphism group.  The Hofer norm $\|\cdot\|$ on $Ham(P,\omega)$ induces a $Ham(P,\omega)$-invariant pseudometric $\delta$ on $\mathcal{L}(S)$, via the formula \[ \delta(S_0,S_1)=\inf\{\|\phi\||\phi\in Ham(P,\omega),\,\phi(S_0)=S_1\}.\]  Here we will consider the case where $(P,\omega)=(T^*N,d\hat{\theta})$ and $S$ is the fiber $T_{x_1}^{*}N$ over a point $x_1\in N$ or, more generally, the \emph{conormal bundle} \[ \nu^*Q=\left\{(x,p)\in T^*N|_Q\left| p|_{T_{x}Q}=0\right.\right\}\] of a submanifold $Q\subset N$.


\begin{theorem}\label{mainlag} Let $(N,g)$, $Q\subset N$, and $x_1\in N\setminus Q$ be as in Theorem \ref{mainham}, and let $U\subset T^*N$ be a neighborhood of the zero section $0_N$.  Then there is a linear map \[  F\co \R^{\infty}\to C^{\infty}(T^*N) \] such that each $F(\vec{a})$ has compact support contained in $U\setminus 0_N$, such that for all $\vec{a}\in \R^{\infty}$ \begin{equation}\label{fv} \max F(\vec{a})-\min F(\vec{a})=osc(\vec{a}), \end{equation} and such that, for some constant $C>0$ and for all $\vec{a},\vec{b}\in \R^{\infty}$ we have \begin{equation} \label{hmorph} \phi_{F(\vec{a})}^{1}\circ \phi_{F(\vec{b})}^{1}=\phi_{F(\vec{a}+\vec{b})}^{1},\end{equation}  \begin{equation}\label{lagest} \delta\left(\phi_{F(\vec{a})}^{1}(T_{x_1}^{*}N),\phi_{F(\vec{b})}^{1}(T_{x_1}^{*}N)\right)\geq \|\vec{a}-\vec{b}\|_{\infty}-C \end{equation} and \begin{equation}\label{lagest2} \delta\left(\phi_{F(\vec{a})}^{1}(\nu^*Q),\phi_{F(\vec{b})}^{1}(\nu^*Q)\right)\geq \|\vec{a}-\vec{b}\|_{\infty}-C. \end{equation}
\end{theorem}

In view of (\ref{fv}), (\ref{hmorph}), and the $Ham$-invariance of $\delta$ we also have \[ \delta\left(\phi_{F(\vec{a})}^{1}(T_{x_1}^{*}N),\phi_{F(\vec{b})}^{1}(T_{x_1}^{*}N)\right)\leq\|\phi_{F(\vec{b}-\vec{a})}^{1}\|\leq osc(\vec{a}-\vec{b})\] and likewise with $\nu^*Q$ in place of $T_{x_1}^{*}N$.  Thus Theorem \ref{mainlag} gives quasi-isometric embeddings of $\R^{\infty}$ into the metric\footnote{The fact that $\delta$ defines a metric, and not just a pseudometric, on $\mathcal{L}(T_{x_1}^{*}N)$ and $\mathcal{L}(\nu^*Q)$ follows from \cite[Remark 4.12]{U12}.} spaces $(\mathcal{L}(T_{x_1}^{*}N),\delta)$ and $(\mathcal{L}(\nu^*Q),\delta)$.  


Theorem \ref{mainlag} quickly implies Theorem \ref{mainham}, as we now show:

\begin{proof}[Proof of Theorem \ref{mainham}, assuming Theorem \ref{mainlag}]  Given $F\co \R^{\infty}\to C^{\infty}(T^*N)$ as in Theorem \ref{mainlag} define $\Phi\co \R^{\infty}\to Ham(T^*N,d\hat{\theta})$ by $\Phi(\vec{a})=\phi_{F(\vec{a})}^{1}$.  By (\ref{hmorph}), $\Phi$ is a homomorphism.    Moreover \[ d(\Phi(\vec{a}),\Phi(\vec{b}))=\|\phi_{F(\vec{a}-\vec{b})}^{1}\|\leq osc(\vec{a}-\vec{b}) \] by (\ref{hmorph}) and (\ref{fv}), proving the upper bound in (\ref{maineq}).  

As for the lower bound, since $\Phi$ is a homomorphism we just need to show that, for all $\vec{a}\in \R^{\infty}$, $\|\Phi(\vec{a})\|\geq \|\vec{a}\|_{\infty}$.  Suppose that this is false, and let $\ep>0$ be such that $\|\Phi(\vec{a})\|\leq \|\vec{a}\|_{\infty}-\ep$.  Where $C$ is the constant in Theorem \ref{mainlag}, choose $m\in \Z_+$ such that $m\ep>C$.  Then since $\Phi$ is a homomorphism and since the Hofer norm satisfies the triangle inequality we would have $\|\Phi(m\vec{a})\|\leq m(\|\vec{a}\|_{\infty}-\ep)<\|m\vec{a}\|_{\infty}-C$.  But this contradicts the fact that, by (\ref{lagest}), $\delta(\Phi(m\vec{a})T_{x_1}^{*}N,T_{x_1}^{*}N)\geq \|m\vec{a}\|_{\infty}-C$.
\end{proof}

\subsection{Related work}
  Theorem \ref{mainlag} is the first result known to the author about the Hofer geometry of the orbit space $\mathcal{L}(T_{x_1}^{*}N)$ of a cotangent fiber under the Hamiltonian diffeomorphism group.  However there are several prior results about the Hofer metric on $Ham(T^*N,d\hat{\theta})$ which imply conclusions closely related to those of Theorem \ref{mainham}, especially in the case that $(N,g)$ has nonpositive sectional curvature.
  
    First of all, \cite{Mili} contains results about the orbit $\mathcal{L}(0_N)$ of the zero section of a the cotangent bundle of an arbitrary compact smooth manifold, nicely complementing our results about the orbits of fibers or of certain conormal bundles.  Namely, \cite[Proposition 1(5,6)]{Mili} shows that, for any compact $N$, the map $C^{\infty}(N)\to \mathcal{L}(0_N)$ which assigns to a smooth function $S\co N\to \R$ the Lagrangian submanifold $graph(dS)$ induces an isometric embedding $(\frac{C^{\infty}(N)}{\R},osc)\to \mathcal{L}(0_N)$, where we write $osc([f])=\max f-\min f$.  Using this, one easily obtains isometric embeddings that map arbitrarily large balls in the normed vector space $(C^{\infty}(N)/\R,osc)$ into $Ham(T^*N,d\hat{\theta})$.  However, in view of the fact that $Ham(T^*N,d\hat{\theta})$ contains only compactly supported diffeomorphisms, it does not seem possible to obtain embeddings of a whole vector space using such a construction; moreover, in order to embed large balls one must use diffeomorphisms with large supports, in contrast to Theorem \ref{mainham} where the supports are fixed.

In \cite{Py08}, Py showed that, whenever a symplectic manifold $(P,\omega)$ contains a $\pi_1$-injective compact Lagrangian submanifold $L$ which admits a Riemannian metric of nonpositive sectional curvature, for any integer $k$ there exists a bi-Lipschitz embedding of $\R^k$ into $Ham(M,\omega)$ (with the relevant constants diverging to $\infty$ with $k$), whose image moreover consists of diffeomorphisms which may be taken to have support in an arbitrary deleted neighborhood of $L$.  Thus in the case that $P=T^*N$ and $L$ is the zero section, Theorem \ref{mainham} represents an improvement on Py's result both in that the embedded subgroups are infinite-dimensional and in that the class of manifolds $N$ to which the theorem applies is more general than just those of nonpositive curvature.  In the case that $P$ is instead compact, a similar improvement of Py's result is a special case of \cite[Theorem 1.1]{U11b}.

More recently, in \cite[Theorem 1.13]{MVZ} it was shown that if a compact manifold $N$ admits a nonsingular closed one-form then $Ham(T^*N,d\hat{\theta})$ admits an isometric embedding of $(C^{\infty}_{c}((0,1)),osc)$.  

Finally, L. Polterovich has pointed out to the author a way of getting a conclusion closely related to that of Theorem \ref{mainham} in the special case that $(N,g)$ has nonpositive curvature (for similar arguments in the contact context see \cite[Examples 2.8-10]{FPR}).  Namely, if for $r>0$ we let $S(r)^{*}N$ denote the radius-$r$ cosphere bundle, one can see from  \cite[Theorem 2.7(iii)]{Gi} that $S(r)^{*}N$ is stably nondisplaceable (since due to the curvature hypothesis $N$ has no contractible closed geodesics).  Then arguing as in \cite[Proposition 7.1.A]{Po} one can show that for any compactly supported smooth function $H\co T^*N\to \R$, the Hofer (pseudo-)norm $\|\tilde{\phi}_H\|$ of the corresponding element $\tilde{\phi}_H$ in the \emph{universal cover} $\widetilde{Ham}(T^*N,d\hat{\theta})$ is greater than or equal to $\max_{S(r)^{*}N}|H|$.  One can then let $r$ vary and consider Hamiltonians of the form $H_f=f\circ \rho$ where $\rho$ is the fiberwise norm on $T^*N$ and where $f\co (0,\infty)\to \R$ is a compactly-supported smooth function.  The assignment $f\mapsto \tilde{\phi}_{H_f}$ is then an embedding  $C^{\infty}_{c}((0,\infty))\hookrightarrow\widetilde{Ham}(T^*N,d\hat{\theta})$ which obeys an estimate analogous to (\ref{maineq}).    In general, it does not seem clear from this argument whether the lower bound in this estimate survives after one projects down from $\widetilde{Ham}(T^*N,d\hat{\theta})$ to $Ham(T^*N,d\hat{\theta})$, though in some isolated cases such as where $N$ is the torus this could likely be deduced along the lines of \cite[Example 2.8]{FPR}.  Also, we should mention that the diffeomorphisms that appear in the image of our embedding $\Phi$ in Theorem \ref{mainham} have the form $\phi_{H_f}$ for certain functions $f$; thus the proof of Theorem \ref{mainham} will show that Polterovich's estimate in the universal cover does survive after projection at least for these special choices of $f$.

\subsection{Summary of the proof}

Our proof of Theorem \ref{mainlag} is based on the properties of the filtrations of the Floer complexes associated to certain pairs of noncompact Lagrangian submanifolds in Liouville manifolds; in our case the Liouville manifold is a cotangent bundle $T^*N$ and the Lagrangian submanifolds are cotangent fibers $T^{*}_{x_1}N$ or, a bit more generally, conormal bundles $\nu^*Q$ of compact submanifolds $Q\subset N$.    Thus if $x_1\notin Q$, for any suitably nondegenerate, compactly supported Hamiltonian $H\co [0,1]\times T^*N\to \R$ we have a Floer complex
$CF(\nu^*Q,T_{x_1}^{*}N;H)$ whose generators correspond to points of $\nu^*Q\cap (\phi_{H}^{1})^{-1}(T_{x_{1}}^{*}N)$.   

Since we are considering compactly supported Hamiltonians (unlike the situation with the ``wrapped'' Floer complexes of, \emph{e.g.}, \cite{AS}), the homology of this Floer complex is trivial, as $\nu^*Q\cap T_{x_1}^{*}N=\varnothing$.    Despite this triviality, we can still obtain significant information (depending, unlike the homology, on the Hamiltonian $H$) from the Floer complex, by means of the \emph{boundary depth}, a quantity that was introduced in a similar context in \cite{U11a} and is denoted here by $B(\nu^*Q,(\phi_{H}^{1})^{-1}(T_{x_{1}}^{*}N))$. The boundary depth gives a quantitative measurement of the nontriviality of the Floer boundary operator with respect to the natural filtration on the Floer complex: it is the smallest nonnegative number $\beta$ such that every element $c$ of the image of the boundary operator has a primitive whose filtration level is at most $\beta$ larger than that of $c$.  As we show in Section \ref{gens} (by arguments quite analogous to ones appearing in \cite[Section 6]{U11b} in the compact case), the boundary depth $B(\nu^*Q,\Lambda)$ is independent of the choice of $H$ with $(\phi_{H}^{1})^{-1}(T_{x_1}^{*}N)=\Lambda$ and moreover, considered as a function of $\Lambda$, is $1$-Lipschitz with respect to the Hofer distance $\delta$ on $\mathcal{L}(T_{x_1}^{*}N)$.  Consequently lower bounds for the boundary depth can give lower bounds on $\delta$ of the sort that appear in (\ref{lagest}).

Thus we prove Theorem \ref{mainlag} by constructing a homomorphism \begin{align*} \R^{\infty}&\to Ham(T^*N,d\hat{\theta}) \\ \vec{a}&\mapsto \phi_{F(\vec{a})}^{1} \end{align*} where $osc(F(\vec{a}))=osc(\vec{a})$ for all $\vec{a}$ while $B(\nu^*Q,(\phi_{F(\vec{a})}^{1})^{-1}(T_{x_{1}}^{*}N))\geq \|a\|_{\infty}-C$ for a constant $C$.  For $F(\vec{a})$ we use a function of the form $f_{\vec{a}}\circ \rho$, where $\rho\co T^*N\to \R$ is the fiberwise norm given by our Riemannian metric $g$ on $N$ and $f_{\vec{a}}\co [0,\infty)\to \R$ is a function whose graph is as in Figure \ref{figure}, with the heights of its minima and maxima given by the coordinates $a_i$ of $\vec{a}$. (More detail on the functions $f_{\vec{a}}$ is given at the start of Section \ref{compute}; our argument requires them to take a rather special form, though this constraint is likely only technical.)
Recall that the function $\rho^2$ on $T^*N$ has Hamiltonian flow given by the geodesic flow of $g$; consequently $f_{\vec{a}}\circ \rho$ will restrict to each cosphere bundle $S(r)^{*}N$ as the time-one map of a reparametrization of the geodesic flow.  A point $(x,p)$ will lie in $\nu^*Q\cap (\phi_{F(\vec{a})}^{1})^{-1}(T_{x_{1}}^{*}N)$ if and only if $x\in Q$, $p\in T_{x}^{*}N$ vanishes on $T_x Q$, and $\pm p$ is  the initial momentum of a geodesic of length $\pm f_{\vec{a}}'(|p|)$ which begins at $x$ and ends at $x_1$.

\begin{figure}\label{figure}
\centering
\includegraphics[width=3.25in]{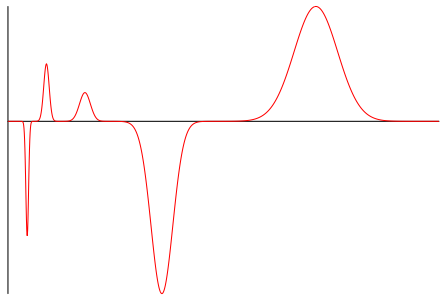}
\caption{The graph of a typical function $f_{\vec{a}}$.  In this case $\vec{a}=(4,-6,1,2,-4,0,0,\ldots)\in \R^{\infty}$, with the coordinates of $\vec{a}$ equal to the heights of the extrema of $f_{\vec{a}}$ as one moves from right to left.}
\end{figure}

Thus a geodesic $\eta$ in $N$ which starts orthogonally to $Q$ and ends at $x_1$ gives rise to several generators for the Floer complex 
$CF(\nu^*Q,T_{x_1}^{*}N;F(\vec{a}))$, one for each value of $r$ such that $|f_{\vec{a}}'(r)|$ is equal to the length of $\eta$.  In order for two generators of the complex to be intertwined by the Floer boundary operator, their corresponding geodesics must be homotopic (through paths from $Q$ to $x_1$), and their Floer-theoretic gradings must differ by $1$.  We compute the Floer-theoretic grading in  Proposition \ref{gradecompQ}; this grading is found to depend in an explicit way on the Morse index of the geodesic $\eta$ and on the signs of both $f_{\vec{a}}'(r)$ and $f_{\vec{a}}''(r)$.
Meanwhile the action of the generator is approximately $f_{\vec{a}}(r)$.

Provided that $(-\min_i a_i)$ is sufficiently large, there will be a generator $c_0$ of $CF(\nu^*Q,T_{x_1}^{*}N;F(\vec{a}))$ whose associated geodesic from $Q$ to $x_1$ has index zero and whose associated value of $r$ obeys $f_{\vec{a}}'(r),f_{\vec{a}}''(r)>0$, such that the action of $c_0$ is approximately $\min_i a_i$.  Under Assumption \ref{assm} (as applies in Theorem \ref{mainlag}), this generator will be a cycle in the Floer complex.  So since the Floer homology vanishes, $c_0$ must also be a boundary.  But Assumption \ref{assm} also implies that the only generators with grading $1$ larger than that of $c_0$ have the property that their associated value of $r$ obeys $f_{\vec{a}}''(r)<0$, with $|f'_{\vec{a}}(r)|$ bounded above independently of $\vec{a}$.  The special form of our functions $f_{\vec{a}}$ then implies that $f_{\vec{a}}(r)$ is bounded below independently of $\vec{a}$.  Thus all possible primitives of $c_0$ have action bounded below by a constant, whereas $c_0$ has action approximately $\min_i(a_i)$.  This leads to an estimate $B(\nu^*Q,(\phi_{F(\vec{a})}^{1})^{-1}(T_{x_{1}}^{*}N))\geq -\min_i(a_i)-C$.  A similar argument, together with a duality property satisfied by the boundary depth, implies that $B(\nu^*Q,(\phi_{F(\vec{a})}^{1})^{-1}(T_{x_{1}}^{*}N))\geq \max_i(a_i)-C$.  Combining these two estimates quickly implies (\ref{lagest}).  Another duality argument yields the similar estimate (\ref{lagest2}) and hence completes the proof of Theorem \ref{mainlag}.

We remark that for many cases of Theorem \ref{mainlag} it would suffice to study Floer complexes of the form $CF(T_{x_0}^{*}N,T_{x_1}^{*}N;H)$ (\emph{i.e.}, to set $Q$ equal to a single point).  However, when $\dim N-\dim Q=2$ one finds from the grading formula in Proposition \ref{gradecompQ} that the argument indicated in the previous paragraph breaks down, and so in particular one cannot set $Q$ equal to a singleton in the case that $N$ is two-dimensional.  So when $N$ is a surface of nonpositive curvature we instead set $Q$ equal to a closed geodesic on $N$ that does not contain $x_1$. Of course, incorporating positive-dimensional $Q$ into our discussion also allows us to obtain the more general results on the Hofer geometry of $\mathcal{L}(\nu^*Q)$ mentioned in (\ref{lagest2}); furthermore, allowing positive-dimensional $Q$ in the hypothesis of Theorem \ref{mainham} leads to the hypothesis being preserved under products with arbitrary compact manifolds, as noted in Remark \ref{product}.

\subsection{Outline of the paper}  In the following Section \ref{gens} we set up the general framework for the Floer complex associated to a Hamiltonian function and a pair of suitable noncompact Lagrangian submanifolds of a Liouville manifold, and we define the boundary depth associated to this complex and prove the basic result Corollary \ref{lip} connecting the boundary depth to Hofer geometry.  Section \ref{geodesic} specializes the discussion to the case in which the Liouville manifold is the cotangent bundle $T^*N$ of a Riemannian manifold, the Hamiltonian is a function of the norm given by the Riemannian metric, and the pair of the Lagrangian submanifolds consists of a conormal bundle $\nu^*Q$ and a cotangent fiber $T_{x_1}^{*}N$.  In this case the Floer complex can be understood in terms of geodesics from $Q$ to $x_1$.  For our main results it is important to work out the grading on the complex, as we do in detail in Section \ref{grading}; this is somewhat easier when $Q$ consists of a single point (which as mentioned earlier suffices for many cases of Theorem \ref{mainlag}), so we treat that case initially, followed by a discussion of the necessary modifications in the more general case.  Next, in Section \ref{assumption}, we introduce Assumption \ref{assm} on the behavior of geodesics connecting $Q$ to $x_1$ and discuss some situations in which it holds. This assumption is a somewhat more general version of the hypotheses of Theorems \ref{mainham} and \ref{mainlag}, and is the most general context in which we are able to prove such results.

Finally, in Section \ref{compute}, we introduce a special class of Hamiltonians for which the grading computation of Section \ref{grading} allows us to get strong lower bounds on the boundary depth under Assumption \ref{assm}, and we prove these lower bounds.  From these lower bounds we deduce Corollary \ref{maincor}, which has Theorem \ref{mainlag} as a special case, thus completing the proofs of our main results.

\subsection*{Acknowledgements} I am grateful to Leonid Polterovich for asking me the question which motivated this work and for various helpful comments.  The work was supported by NSF grant DMS-1105700.


  
  


\section{Generalities on the Floer complex and the boundary depth}\label{gens}
Let $(M,\theta)$ be a Liouville domain; thus $M$ is a compact manifold with boundary, with $\theta\in \Omega^1(M)$ such that $d\theta$ is symplectic, and the vector field $Z$ characterized by the property that $\iota_Zd\theta=\theta$ points outward along $\partial M$.  This implies that $\alpha:=\theta|_{\partial M}$ is a contact form. If for $t\geq 0$ we denote by $\zeta^t\co M\to M$ the time-$t$ flow of the vector field $-Z$, then for some $\ep_0>0$ the map $\Psi\co(1-\ep_0,1]\times \partial M\hookrightarrow M$ defined by $\Psi(r,m)=\zeta^{-\ln r}(m)$  is an embedding with the property that $\Psi^*\theta=r\alpha$.  Then the Liouville completion $(\hat{M},\hat{\theta})$ may be defined as \[ \hat{M}=\frac{M\cup \left((1-\ep_0,\infty)\times \partial M\right)}{m\sim \Psi^{-1}(m)\mbox{ for }m\in Im(\Psi)} \qquad \hat{\theta}|_M=\theta,\,\hat{\theta}|_{(1-\ep_0,\infty)\times\partial M}=r\alpha.\]  In particular $(\hat{M},d\hat{\theta})$ is a symplectic manifold.

Hereinafter we will identify $(1-\ep_0,1]\times \partial M$ with a subset of $M$; it should be understood that we are using $\Psi$ to make this identification.

\begin{dfn} A \emph{filling Lagrangian} in $(M,\theta)$ is a compact $\frac{\dim M}{2}$-dimensional submanifold $L$ of $M$ with boundary $\partial L$, such that: \begin{itemize} \item[(i)] $\partial L\subset \partial M$.
\item[(ii)] For some $\ep_L$ with $0<\ep_L<\ep_0$ we have (with respect to $\Psi$) \[ L\cap \left([1-\ep_L,1]\times\partial M\right)=[1-\ep_L,1]\times \partial L \]
\item[(iii)] For some smooth function $h\co L\to \R$ which vanishes to infinite order along $\partial L$, we have $\theta|_{L}=dh$.
\end{itemize}
\end{dfn}

Of course, (iii) implies that $L$ is an exact Lagrangian submanifold, and (ii) and (iii) together imply that in fact $\theta|_{L\cap \left([1-\ep_L,1]\times\partial M\right)}=0$.  Any filling Lagrangian $L\subset M$ gives rise to a properly embedded, exact Lagrangian submanifold without boundary $\hat{L}\subset \hat{M}$, namely  \[ \hat{L}=L\cup([1,\infty)\times \partial M).\]
We will refer to $\hat{L}$ as the \emph{completion} of $L$.

 The standard example of a filling Lagrangian occurs when $M=D^*N$ is the disk cotangent bundle of some smooth manifold $N$ and $\theta$ is the tautological one-form; then we may take $L$ equal to the disk conormal bundle of any compact submanifold of $N$.  (Indeed in this case $\theta|_L=0$.) The name ``filling Lagrangian'' refers to the fact that $L$ provides what is sometimes called a filling for the Legendrian submanifold $\partial L\subset \partial M$.
 
The results of this paper make use of a Floer complex $\left(CF_{\frak{c}}(\hat{L}_0,\hat{L}_1;H),\partial_{\mathbb{J}}\right)$ that may be associated to data of the following sort:

\begin{itemize} \item[(I)] Completions $\hat{L}_0,\hat{L}_1\subset \hat{M}$ of a pair of filling Lagrangians $L_0,L_1\subset M$ such that $\partial L_0\cap \partial L_1=\varnothing$.
\item[(II)] A smooth, compactly supported function $H\co [0,1]\times \hat{M}\to \R$, such that the time-one map $\phi_{H}^{1}$ of the Hamiltonian flow of $H$ has the property that $\phi_{H}^{1}(\hat{L}_0)$ is transverse to $\hat{L}_1$.
\item[(III)] An element $\frak{c}\in \pi_0(\mathcal{P}(\hat{L}_0,\hat{L}_1))$, where $\mathcal{P}(\hat{L}_0,\hat{L}_1)$ is the space of piecewise-smooth paths $\gamma\co [0,1]\to\hat{M}$ such that $\gamma(0)\in \hat{L}_0,\gamma(1)\in \hat{L}_1$.
\item[(IV)] A suitably generic path $\mathbb{J}=\{J_t\}_{t\in [0,1]}$ of $d\hat{\theta}$-compatible almost complex structures on $\hat{M}$ such that, for some real number $R_{\mathbb{J}}>1-\ep_0$, we have $J_t=J_0$ for all $t$ on $[R_{\mathbb{J}},\infty)\times \partial M$, and  $\alpha\circ J_0=dr$ on $[R_{\mathbb{J}},\infty)\times \partial M$.
\end{itemize}  

Given $\hat{L}_0,\hat{L}_1,$ and $\frak{c}\in \pi_0(\mathcal{P}(\hat{L}_0,\hat{L}_1))$, fix (independently of $H$ and $J$) a path $\gamma_{\frak{c}}\co [0,1]\to \hat{M}$ which represents the class $\frak{c}$,  and fix a symplectic trivialization $\tau_{\frak{c}}\co \gamma_{\frak{c}}^{*}TM\to [0,1]\times \R^{2n}$ which sends $T_{\gamma(0)}\hat{L}_0$ to $\{0\}\times\R^n\times\{\vec{0}\}$ and $T_{\gamma(1)}\hat{L}_1$ to $\{1\}\times\R^n\times\{\vec{0}\}$.  Formally speaking, $\left(CF_{\frak{c}}(L_0,L_1;H),\partial_{\mathbb{J}}\right)$ is the Morse complex of the function $\mathcal{A}_H\co \frak{c}\to \R$ (where again $\frak{c}$ is a path component of $\mathcal{P}(\hat{L}_0,\hat{L}_1)$), defined by \[ \mathcal{A}_H(\gamma)=-\int_{[0,1]^{2}}u^*d\hat{\theta}+\int_{0}^{1}H(t,\gamma(t)) dt \] where $u\co [0,1]^{2}\to \hat{M}$ is a smooth map with $u(0,t)=\gamma_{\frak{c}}(t)$, $u(1,t)=\gamma(t)$, $u(s,0)\in \hat{L}_0$, and $u(s,1)\in \hat{L}_1$ for all $s,t\in [0,1]$ (the fact that the $\hat{L}_i$ are exact makes $\mathcal{A}_H$ independent of the choice of such $u$).   A critical point $\gamma$ of $\mathcal{A}_H$ has $\gamma'(t)=X_H(t,\gamma(t))$ for all $t$; thus critical points of $\mathcal{A}_H$ are in natural bijection with intersection points $\gamma(0)\in \hat{L}_0\cap (\phi_{H}^{1})^{-1}(\hat{L}_1)$.  Since $\phi_{H}^{1}$ is equal to the identity outside of a compact set and since $\hat{L}_0\cap \hat{L}_1\cap ([1,\infty)\times\partial M)=\varnothing$, the assumed transversality of $\phi_{H}^{1}(\hat{L}_0)$ and $\hat{L}_1$ implies that there are only finitely many such points $\gamma(0)$, and so only finitely many elements $\gamma$ in the critical locus $Crit(\mathcal{A}_H)$ of $\mathcal{A}_H$.

  There is a homomorphism $\mu_{\frak{c}}\co \pi_1(\frak{c},\gamma_{\frak{c}})\to \Z$ which assigns to the homotopy class of a smooth map $u\co (S^1\times [0,1],S^1\times\{0\},S^1\times\{1\})\to (\hat{M},\hat{L}_0,\hat{L}_1)$ the difference of the Maslov indices of the loops of Lagrangian subspaces $(u|_{S^1\times\{1\}})^{*}T\hat{L}_{1}$ and $(u|_{S^1\times\{0\}})^{*}T\hat{L}_{0}$ with respect to an arbitrary symplectic trivialization of $u^*T\hat{M}$.  Let $N_{\frak{c}}\in \Z$ be the positive generator of the image of $\mu_{\frak{c}}$ (if $\mu_{\frak{c}}=0$---as will be the case in our main application---we set $N_{\frak{c}}=0$).

Given these data, to each $\gamma\in Crit(\mathcal{A}_H)$ representing the class $\frak{c}$ we may associate a Maslov-type  index $\mu(\gamma)\in \Z/N_{\frak{c}}\Z$.  Namely, choose an arbitrary piecewise $C^1$ map $v\co [0,1]^2\to \hat{M}$ such that $v(0,t)=\gamma_{\frak{c}}(t)$, $v(1,t)=\gamma(t)$, $v(s,0)\in\hat{L}_0$, and $v(s,1)\in \hat{L}_1$ for all $s,t\in[0,1]$, and choose a symplectic trivialization of $v^*TM$ which restricts over $\{0\}\times[0,1]$ to the trivialization $\tau_{\frak{c}}$ and which identifies each $T_{v(s,i)}\hat{L}_i$ with $\R^{n}\times\{\vec{0}\}$.  With respect to this trivialization, the path $\Gamma(t)=(\phi_{H}^{t})_{*}T_{\gamma(0)}\hat{L}_0$ is a path in the Lagrangian Grassmannian such that $\Gamma(0)=\R^n\times\{\vec{0}\}$ and $\Gamma(1)$ is transverse to $\R^n\times\{\vec{0}\}$.  We let \begin{equation}\label{gradeconv} \mu(\gamma)=\frac{n}{2}-\mu_{RS}(\Gamma,\R^n\times\{\vec{0}\}) \mod N_{\frak{c}}\end{equation} where  $\mu_{RS}(\Gamma,\R^n\times\{\vec{0}\})$ is the Robbin--Salamon--Maslov index of the path $\Gamma$ with respect to $\R^n\times\{\vec{0}\}$ as defined in \cite[Section 2]{RS}.  Since $\mu_{RS}(\Gamma,\R^n\times\{\vec{0}\})$ receives an initial contribution of one-half of the signature of the crossing form of $\Gamma$ with $\R^n\times\{\vec{0}\}$ at $t=0$, and this signature is congruent to $n$ modulo $2$, $\frac{n}{2}-\mu_{RS}(\Gamma,\R^n\times\{\vec{0}\})$ is indeed an integer; the subsequent reduction modulo $N_{\frak{c}}$ removes the dependence of $\mu(\gamma)$ on the homotopy $v$ from $\gamma_{\frak{c}}$ to $\gamma$, so that $\mu(\gamma)\in\frac{\Z}{N_{\frak{c}}\Z}$ indeed depends only on $\gamma$ (given the choices of $\gamma_{\frak{c}}$ and $\tau_{\frak{c}}$ that were made at the outset).
Our normalization is designed so that if $L_0=L_1$ is the zero section of the cotangent bundle $T^*N$ and $H$ is the pullback of a $C^2$-small Morse function on $N$, so that elements of $Crit(\mathcal{A}_H)$ are constant paths $\gamma_p$ at critical points $p$ of $H$, then the Maslov index $\mu(\gamma_p)$ coincides with the Morse index of $p$.

Now for $k\in \Z/N_{\frak{c}}\Z$ let $CF_{\frak{c},k}(\hat{L}_0,\hat{L}_1;H)$ be the $\Z/2\Z$-vector space generated by those $\gamma\in Crit(\mathcal{A}_H)$ representing the homotopy class $\frak{c}$ with $\mu(\gamma)=k$.  For $\lambda\in\R$ let $CF^{\lambda}_{\frak{c},k}(\hat{L}_0,\hat{L}_1;H)\leq CF_{\frak{c},k}(\hat{L}_0,\hat{L}_1;H)$ be the subspace generated by those $\gamma$ which additionally have $\mathcal{A}_H(\gamma)\leq \lambda$.

Given a path of almost complex structures $\mathbb{J}=\{J_t\}_{t\in [0,1]}$ as in (IV) above we consider the associated negative gradient flow equation for $\mathcal{A}_H$, for a map $u\co \R\times [0,1]\to \hat{M}$:\begin{equation}\label{floereq} \frac{\partial u}{\partial s}+J_t(u(s,t))\left(\frac{\partial u}{\partial t}-X_{H}(t,u(s,t))\right)=0\quad u(s,0)\in \hat{L}_0,\,u(s,1)\in \hat{L}_1\mbox{ for all }s\in \R.\end{equation}  For any finite-energy solution to (\ref{floereq}) there are $\gamma_{\pm}\in Crit(\mathcal{A}_H)$ such that $u(s,\cdot)\to \gamma_{\pm}$ uniformly in $t$ as $s\to\pm\infty$.  Note that our assumptions on $H$ and on the $L_i$ ensure that $\gamma_{-}$ and $\gamma_+$ both have image contained in $\hat{M}\setminus ([R_H,\infty)\times \partial M)$, where $R_H\geq R_{\mathbb{J}}$ is chosen so large that $H$ vanishes identically on $[0,1]\times [R_H,\infty)\times \partial M$.  In fact, our assumptions on $\mathbb{J}$ together with a maximum principle ensure that any finite-energy solution $u$ to (\ref{floereq}) must have image contained entirely in $\hat{M}\setminus ((R_H,\infty)\times \partial M)$, as can be seen for instance from \cite[Theorem 2.1]{OhNC} or \cite[Sections 7c,7d]{AS}.

In view of this maximum principle and of the fact that bubbling is prevented by the exactness of the symplectic form on $\hat{M}$ and of the Lagrangian submanifolds $\hat{L}_0,\hat{L}_1$, the standard construction of the Floer boundary operator as in \cite{F}, \cite{OhI} gives for suitably generic $\mathbb{J}$ a map $\partial_{\mathbb{J}}\co CF_{\frak{c}}(\hat{L}_0,\hat{L}_1;H)\to CF_{\frak{c}}(\hat{L}_0,\hat{L}_1;H)$ which counts finite-energy index-one solutions to (\ref{floereq}), and satisfies  $\partial_{\mathbb{J}}\circ \partial_{\mathbb{J}}=0$.  Moreover for each $\lambda\in \R$ and $k\in \Z/N_{\frak{c}}\Z$  we have \[ \partial_{\mathbb{J}}(CF_{\frak{c},k}^{\lambda}(\hat{L}_0,\hat{L}_1;H))\subset CF_{\frak{c},k-1}^{\lambda}(\hat{L}_0,\hat{L}_1;H).\] Thus $(CF_{\frak{c}}(\hat{L}_0,\hat{L}_1;H),\partial_{\mathbb{J}})$ is  a $\Z/N_{\frak{c}}\Z$-graded, $\R$-filtered chain complex of $\mathbb{Z}/2\Z$-vector spaces.

For any element $c=\sum_i a_i\gamma_i\in CF_{\frak{c}}(\hat{L}_0,\hat{L}_1;H)$ (where $a_i\in\Z/2\Z$ and $\gamma_i\in Crit(\mathcal{A}_H)$), write \[ \ell(c)=\max\{\mathcal{A}_H(\gamma_i)|a_i\neq 0\},\] where the maximum of the empty set is defined to be $-\infty$.  Thus $\ell(\partial_{\mathbb{J}}c)\leq \ell(c)$ for all $c\in CF_{\frak{c}}(\hat{L}_0,\hat{L}_1;H)$.

Choose a monotone increasing smooth function $\beta\co \R\to [0,1]$ such that $\beta(s)=0$ for $s\leq -1$ and $\beta(s)=1$ for $s\geq 1$.
Given two pairs $(H_-,\mathbb{J}_-=\{J_{-,t}\})$, $(H_+,\mathbb{J}_+=\{J_{+,t}\})$ as in (II), (IV), define $H\co \R\times [0,1]\times M\to\R$ by $H(s,t,m)=\beta(s)H_-(t,m)+(1-\beta(s))H_+(m)$.  For a suitably generic family of almost complex structures $J_{s,t}$ as in (IV) which interpolate between $J_{-,t}$ and $J_{+,t}$, counting isolated solutions $u\co \R\times[0,1]\to\hat{M}$ to \[ \frac{\partial u}{\partial s}+J_{s,t}(u(s,t))\left(\frac{\partial u}{\partial t}-X_{H(s,\cdot)}(t,u(s,t))\right)=0,\quad u(\R\times \{0\})\subset \hat{L}_0,u(\R\times\{1\})\subset \hat{L}_1 \] gives rise to a chain map \[ \Phi_{H_-H_+}\co CF_{\frak{c}}(\hat{L}_0,\hat{L}_1;H_-)\to CF_{\frak{c}}(\hat{L}_0,\hat{L}_1;H_+).\]  A standard estimate (see, \emph{e.g.}, \cite[p. 564]{OhI}, noting the different sign conventions) gives \[   \ell(\Phi_{H_-H_+}c)\leq \ell(c)+\|H_+-H_-\| \]  for all $c\in CF_{\frak{c}}(\hat{L}_0,\hat{L}_1;H_-)$, where $\|\cdot\|$ denotes the Hofer norm \[ \|H\|=\int_{0}^{1}\left(\max_{\hat{M}}H(t,\cdot)-\min_{\hat{M}}H(t,\cdot)\right)dt.\]  Moreover by considering appropriate homotopies of homotopies (see \emph{e.g.}, \cite[Proposition 2.2]{U11a} and the text preceding it for details in the essentially identical context of Hamiltonian Floer theory), one obtains maps $\mathcal{K}_{\pm}\co CF_{\frak{c}}(\hat{L}_0,\hat{L}_1;H_{\pm})\to CF_{\frak{c}}(\hat{L}_0,\hat{L}_1;H_{\pm})$ such that \[ \Phi_{H_+H_-}\circ \Phi_{H_-H_+}-1=\partial_{\mathbb{J}_-}\circ\mathcal{K}_-+\mathcal{K}_-\circ \partial_{\mathbb{J}_-},\quad \Phi_{H_-H_+}\circ \Phi_{H_+H_-}-1=\partial_{\mathbb{J}_+}\circ\mathcal{K}_++\mathcal{K}_+\circ \partial_{\mathbb{J}_+} \] and, for all $c\in CF_{\frak{c}}(\hat{L}_0,\hat{L}_1;H_{\pm})$, \[ \ell(\mathcal{K}_{\pm}c)\leq \ell(c)+\|H_+-H_-\|.\]  In other words, in the language of \cite[Definition 3.7]{U11b}, the $\mathbb{Z}/N_{\frak{c}}\Z$-graded, $\R$-filtered complexes $(CF_{\frak{c}}(\hat{L}_0,\hat{L}_1;H_-),\partial_{\mathbb{J}_-})$ and $(CF_{\frak{c}}(\hat{L}_0,\hat{L}_1;H_+),\partial_{\mathbb{J}_+})$ are ``$\|H_+-H_-\|$-quasiequivalent.''

We define the boundary depth \[ \beta_{\frak{c},k}(\hat{L}_0,\hat{L}_1;H)=\inf\left\{\beta\geq 0\left|(\forall \lambda\in \R)\left((Im\partial_{\mathbb{J}})\cap CF_{\frak{c},k}^{\lambda}(\hat{L}_0,\hat{L}_1;H)\subset \partial_{\mathbb{J}}\left(CF_{\frak{c},k+1}^{\lambda+\beta}(\hat{L}_0,\hat{L}_1;H)\right)\right) \right.\right\}.\] \cite[Proposition 3.8]{U11b} shows that we have \begin{equation}\label{betacont} |\beta_{\frak{c},k}(\hat{L}_0,\hat{L}_1;H_+)-\beta_{\frak{c},k}(\hat{L}_0,\hat{L}_1;H_-)|\leq \|H_+-H_-\|.\end{equation}  (In particular $\beta_{\frak{c},k}(\hat{L}_0,\hat{L}_1;H)$ is independent of the generic family of almost complex structures as in (IV) used to define it.)  \emph{A priori}, $\beta_{\frak{c},k}(\hat{L}_0,\hat{L}_1;H)$ has only been defined when $H$ obeys (II), but (\ref{betacont}) allows us to extend this definition continuously to arbitrary smooth (or even just continuous) compactly supported functions $H\co [0,1]\times \hat{M}\to \R$.

Now generators for the Floer complex $CF_{\frak{c}}(\hat{L}_0,\hat{L}_1;H)$ correspond to certain intersection points of $\hat{L}_0$ with $(\phi_{H}^{1})^{-1}(\hat{L}_1)$, while the definition of the filtration on the complex and hence of the boundary depth $\beta_{\frak{c},k}(\hat{L}_0,\hat{L}_1;H)$ at least appear to depend on the particular Hamiltonian function $H$ generating the time-one map $\phi_{H}^{1}$.  In fact we will presently see that, much like the situation in \cite[Section 6]{U11b}, modulo shifts in the homotopy and grading data $\frak{c}$ and $k$ the boundary depth actually only depends on the Lagrangian submanifolds $\hat{L}_0$ and $(\phi_{H}^{1})^{-1}(\hat{L}_1)$.

To set this up, consider any Hamiltonian isotopy $\psi_t\co \hat{M}\to \hat{M}$ ($0\leq t\leq 1$) with $\psi_0=1_{\hat{M}}$ and such that $\psi_1(\hat{L}_1)=\hat{L}_1$. For a path $\gamma\in \mathcal{P}(\hat{L}_0,\hat{L}_1)$, define a new path $\Psi\gamma\co [0,1]\to \hat{M}$ by \[ (\Psi\gamma)(t)=\psi_t(\gamma(t)).\]  The fact that $\psi_1(\hat{L}_1)=\hat{L}_1$ implies that $\Psi\gamma\in \mathcal{P}(\hat{L}_0,\hat{L}_1)$.  This induces a map \[ \Psi_*\co \pi_0(\mathcal{P}(\hat{L}_0,\hat{L}_1))\to \pi_0(\mathcal{P}(\hat{L}_0,\hat{L}_1)) \] defined by $\Psi_*[\gamma]=[\Psi\gamma]$.
Where $Ham_{\hat{L}_1}$ is the subgroup of $Ham(\hat{M},d\hat{\theta})$ consisting of Hamiltonian diffeomorphisms which preserve $\hat{L}_1$, the map $\Psi_*$ evidently depends only on the relative homotopy class $[\Psi]\in \pi_1(Ham(\hat{M},d\hat{\theta}),Ham_{\hat{L}_1},1_M)$.  Thus we have an action of $\pi_1(Ham(\hat{M},d\hat{\theta}),Ham_{\hat{L}_1},1_M)$ on $\pi_0(\mathcal{P}(\hat{L}_0,\hat{L}_1))$.

\begin{remark} In our main application it will hold that, for any compact set $K\subset \hat{M}$, every class in $\pi_0(\mathcal{P}(\hat{L}_0,\hat{L}_1))$ is represented by a path which is disjoint from $K$.  Since our Hamiltonian isotopies are compactly supported it follows that in this case the action of $\pi_1(Ham(\hat{M},d\hat{\theta}),Ham_{L_1},1_M)$ on $\pi_0(\mathcal{P}(\hat{L}_0,\hat{L}_1))$ is trivial.
\end{remark}

If two Hamiltonians $G,H\co [0,1]\times \hat{M}\to \R$ have the property that $(\phi_{H}^{1})^{-1}(\hat{L}_1)=(\phi_{G}^{1})^{-1}(\hat{L}_1)$, define $\psi_t=\phi_{H}^{t}\circ (\phi_{G}^{t})^{-1}$, so that $\psi_1(\hat{L}_1)=\hat{L}_1$.  If $F\co [0,1]\times \hat{M}\to \R$ is the Hamiltonian generating the isotopy $\{\psi_t\}$, we will have \begin{equation}\label{ghf} G(t,m)=(H-F)(t,\psi_t(m))\end{equation} for $(t,m)\in [0,1]\times\hat{M}$.

\begin{prop}\label{shiftiso} If $G,H,F,\psi_t$ are as above and if $\phi_{H}^{1}(\hat{L}_0)$ is transverse to $\hat{L}_1$, then for each $\frak{c}\in \pi_0(\mathcal{P}(\hat{L}_0,\hat{L}_1))$ there are $\mu_{\frak{c}}\in \Z/N_{\frak{c}}\Z$, $\lambda_{\frak{c}}\in \R$, and paths of almost complex strucures $\mathbb{J}$, $\mathbb{J}'$ as in (IV) such that there is an isomorphism of chain complexes \[ \Phi\co (CF_{\frak{c}}(\hat{L}_0,\hat{L}_1;G),\partial_{\mathbb{J}})\to (CF_{\Psi_*\frak{c}}(\hat{L}_0,\hat{L}_1;H),\partial_{\mathbb{J}'}) \] with the property that, for each $k\in \Z/N_{\frak{c}}\Z$ and $\lambda\in \R$, \[ \Phi\left(CF_{\frak{c},k}^{\lambda}(\hat{L}_0,\hat{L}_1;G)\right)=CF_{\Psi_*\frak{c},k+\mu_{\frak{c}}}^{\lambda+\lambda_{\frak{c}}}(\hat{L}_0,\hat{L}_1;H).\]
\end{prop}

\begin{proof} The proof is essentially the same as that of \cite[Proposition 6.2]{U11b}.  First note that a path $\gamma\co[0,1]\to \hat{M}$ represents the class $\frak{c}\in \pi_0(\mathcal{P}(\hat{L}_0,\hat{L}_1))$ and belongs to $Crit(\mathcal{A}_G)$ if and only if the curve $\Psi\gamma$ represents the class $\Psi_{*}\frak{c}$ and belongs to $Crit(\mathcal{A}_H)$.  So we can define the map $\Phi$ on generators by setting $\Phi(\gamma)=\Psi\gamma$; obviously $\Phi$ is an isomorphism of vector spaces.

For general maps $u\co \R\times [0,1]\to\hat{M}$ and $K\co [0,1]\times \hat{M}\to \R$ and families of almost complex structures $\mathbb{J}=\{J_t\}_{0\leq t\leq 1}$ we denote \[ \delbar_{\mathbb{J},K}u=\frac{\partial u}{\partial s}+J_t\left(\frac{\partial u}{\partial t}-X_K\right);\] this is a section of $u^*T\hat{M}$.  If, given $\mathbb{J}$ as in (IV), we define $\mathbb{J}'=\{J'_t\}_{0\leq t\leq 1}$ by $J'_t=(\psi_{t})_* J_t(\psi_{t}^{-1})_*$, a routine computation shows that, for all $(s,t)\in \R\times S^1$,\[  (\psi_{t})_*\left((\delbar_{\mathbb{J},G}u)_{(s,t)}\right)=\left(\delbar_{\mathbb{J}',H}(\Psi u)\right)_{(s,t)}\] where $(\Psi u)(s,t)=\psi_t(u(s,t))$.  Of course we have $u(\R\times\{i\})\subset \hat{L}_i$ if and only if $(\Psi u)(\R\times \{i\})\subset \hat{L}_i$, since $\psi_0$ is the identity and $\psi_1(\hat{L}_1)=\hat{L}_1$.  So the map $u\mapsto \Psi u$ sends the moduli spaces of flowlines used to define the differential $\partial_{\mathbb{J}}$ on $CF_{\frak{c}}(\hat{L}_0,\hat{L}_1;G)$  bijectively to the corresponding (via $\Phi$) moduli spaces used to define $\partial_{\mathbb{J}'}$ on $CF_{\frak{c}}(\hat{L}_0,\hat{L}_1;H)$ (and moreover preserves Fredholm regularity of these solutions).  This proves that $\Phi$ is an isomorphism of chain complexes, provided of course that $\mathbb{J}$ has been chosen generically so as to ensure that $\partial_{\mathbb{J}}$ is well-defined and satisfies the usual properties.

It remains to prove the statement about the effect of $\Phi$ on gradings and filtrations.  Recall that, to define the action functionals $\mathcal{A}_G,\mathcal{A}_H$, we have chosen representatives $\gamma_{\frak{c}}$ for each class $\frak{c}\in \pi_0(\mathcal{P}(\hat{L}_0,\hat{L}_1))$. For each $\frak{c}\in \pi_0(\mathcal{P}(\hat{L}_0,\hat{L}_1)$, choose once and for all a map $w\co [0,1]\times[0,1]\to \hat{M}$ such that $w([0,1]\times\{i\})\subset \hat{L}_i$ for $i=0,1$, while $w(0,\cdot)=\gamma_{\Psi_*\frak{c}}$ and $w(1,\cdot)=\Psi\gamma_{\frak{c}}$.  

If $\gamma$ is a representative of $\frak{c}$, let $v\co[0,1]\times[0,1] \to \hat{M}$
be such that $v([0,1]\times\{i\})\subset \hat{L}_i$ for $i=0,1$, $v(0,\cdot)=\gamma_{\frak{c}}$, and $v(1,\cdot)=\gamma$.  Also define $(\Psi v)(s,t)=\psi_t(v(s,t))$.  Then concatenating $w$ and $\Psi v$ gives a homotopy from $\gamma_{\Psi_*\frak{c}}$ to $\Psi\gamma$.  We therefore have, using (\ref{ghf}) and the fact that the $\psi_t$ are symplectomorphisms \begin{align*}
\mathcal{A}_H(\Psi\gamma)&=-\int_{[0,1]^2}w^{*}d\hat{\theta}-\int_{[0,1]^2}(\Psi v)^{*}d\hat{\theta}+\int_{0}^{1}H(t,\psi_t(\gamma(t)))dt \\&= -\int_{[0,1]^2}w^{*}d\hat{\theta}-\int_{[0,1]^2}v^*d\hat{\theta}-\int_{0}^{1}\int_{0}^{1}(dF)_{\psi_t(v(s,t))}\left(\frac{\partial (\Psi v)}{\partial s}\right)dsdt +\int_{0}^{1}H(t,\psi_t(\gamma(t)))dt 
\\&=-\int_{[0,1]^2}w^{*}d\hat{\theta}-\int_{[0,1]^2}v^*\hat{\theta}-\int_{0}^{1}(F(t,\psi_t(\gamma(t)))-F(t,\Psi\gamma_{\frak{c}}))dt+\int_{0}^{1}H(t,\psi_t(\gamma(t)))dt
\\&=\mathcal{A}_F(\gamma_{\Psi_*\frak{c}})-\int_{[0,1]^2}v^*d\hat{\theta}+\int_{0}^{1}(H-F)(t,\psi_t(\gamma(t)))dt \\&=
\mathcal{A}_F(\gamma_{\Psi_*\frak{c}})+\mathcal{A}_G(\gamma).\end{align*}
This proves that our chain isomorphism $\Phi$ obeys $\ell(\Phi c)=\ell(c)+\mathcal{A}_F(\gamma_{\Psi_*\frak{c}})$ for all $c\in CF_{\frak{c}}(\hat{L}_0,\hat{L}_1;H)$, so we may set $\lambda_{\frak{c}}=\mathcal{A}_F(\gamma_{\Psi_*\frak{c}})$ in the statement of the proposition.

As for the gradings, by using a trivialization over the concatenation of the above maps $w$ and $\Psi v$ to compute $\mu(\Psi\gamma)$, one may verify that $\mu(\Psi\gamma)-\mu(\gamma)$ is given as follows.  Let $\tau^{w}\co w^{*}T\hat{M}\to [0,1]^2\times \R^{2n}$ be a symplectic trivialization which restricts to our fixed trivialization $\tau_{\Psi_*\frak{c}}$ over $\gamma_{\Psi_*{\frak{c}}}$ and which sends $(w^{*}T\hat{L}_i)_{(s,i)}$ to $\R^{n}\times \{\vec{0}\}$ for all $s\in [0,1]$ and $i=0,1$.  We then have a Lagrangian subbundle $\mathcal{L}_w=(\tau^{w})^{-1}(\{1\}\times [0,1]\times\R^n\times\{\vec{0}\})$ of $(\Psi\gamma_{\frak{c}})^{*}T\hat{M}$, which coincides over $i=0,1$ with $T_{(\Psi\gamma)(i)}\hat{L}_i$.  Meanwhile, we have a trivialization $\Psi_*\tau_{\frak{c}}\co (\Psi\gamma_{\frak{c}})^{*}T\hat{M}\to [0,1]\times \R^{2n}$ obtained by pushing forward the fixed trivialization $\tau_{\frak{c}}$ of $\gamma_{\frak{c}}^{*}T\hat{M}$ in the obvious way, and  $\Psi_*\tau_{\frak{c}}(T_{(\Psi\gamma)(i)}\hat{L}_i)=\R^{n}\times\{\vec{0}\}$.  $\Psi_*\tau_{\frak{c}}(\mathcal{L}_w)$ thus defines a loop of Lagrangian subspaces of $\R^{2n}$, and using the catenation property of the Robbin--Salamon--Maslov index $\mu(\Psi\gamma)-\mu(\gamma)$ can be seen to agree with the Maslov index of this loop.  So we may define $\mu_{\frak{c}}$ as the mod $N_{\frak{c}}$ reduction of the Maslov index of $\Psi_*\tau_{\frak{c}}(\mathcal{L}_w)$, completing the proof.
\end{proof}

Recall from the introduction that 
for any closed subset $S\subset \hat{M}$ we define \[ \mathcal{L}(S)=\{\phi(S)|\phi\in Ham(\hat{M},d\hat{\theta})\},\] and for $S_1,S_2\in \mathcal{L}(S)$ \[ \delta(S_1,S_2)=\inf\left\{\left.\int_{0}^{1}\left(\max_{\hat{M}} H(t,\cdot)-\min_{\hat{M}} H(t,\cdot)\right)dt\right|\phi_{H}^{1}(S_1)=S_2,\,H\in C^{\infty}_{c}([0,1]\times\hat{M})  \right\}.\] where $C^{\infty}_{c}([0,1]\times\hat{M})$ denotes the space of compactly-supported smooth real-valued functions on $[0,1]\times \hat{M}$.

If $L_0,L_1$ are two filling Lagrangians in $M$, and if $\Lambda\in\mathcal{L}(\hat{L}_1)$, define \begin{equation}\label{Bdef} B(\hat{L}_0,\Lambda)=\sup_{\frak{c},k}\beta_{\frak{c},k}(\hat{L}_0,\hat{L}_1;H)\mbox{ for any $H\in C^{\infty}_{c}([0,1]\times\hat{M})$ with }(\phi_{H}^{1})^{-1}(\hat{L}_1)=\Lambda.\end{equation}

\begin{cor}\label{lip} $B(\hat{L}_0,\Lambda)$ is independent of the choice of $H$ used to define it. Moreover, for any $\Lambda,\Lambda'\in \mathcal{L}(\hat{L}_1)$, \[ \delta(\Lambda,\Lambda')\geq |B(\hat{L}_0,\Lambda)-B(\hat{L}_0,\Lambda')|.\]
\end{cor} 

\begin{proof}
Suppose for the moment that $\hat{L}_0$ is transverse to $\Lambda$, so that if we choose any $H\in C^{\infty}_{c}([0,1]\times \hat{M})$ such that $(\phi_{H}^{1})^{-1}(\hat{L}_1)=\Lambda$ then we have a well-defined Floer complex $(CF(\hat{L}_0,\hat{L}_1;H),\partial_{\mathbb{J}})$, with the boundary depth $\beta_{\frak{c},k}(\hat{L}_0,\hat{L}_1;H)$ independent of the choice of $\mathbb{J}$.  If $G\in C^{\infty}_{c}([0,1]\times\hat{M})$ also has $(\phi_{G}^{1})^{-1}(\hat{L}_1)=\Lambda$, it follows directly from Proposition \ref{shiftiso} that for each $\frak{c}\in \pi_0(\mathcal{P}(\hat{L}_0,\hat{L}_1))$ there is $\mu_{\frak{c}}\in \Z/N_{\frak{c}}\Z$ such that, for all $k\in\Z/N_{\frak{c}}\Z$, $\beta_{\frak{c},k}(\hat{L}_0,\hat{L}_1;G)=\beta_{\Psi_*\frak{c},k+\mu_{\frak{c}}}(\hat{L}_0,\hat{L}_1;H)$ (where $\Psi_*\co \pi_0(\mathcal{P}(\hat{L}_0,\hat{L}_1))\to\pi_0(\mathcal{P}(\hat{L}_0,\hat{L}_1))$ is a bijection).  In particular the suprema of $\beta$ over all  $(\frak{c},k)$ are identical for $G$ and $H$, confirming that $B(\hat{L}_0,\Lambda)$ is independent of the choice of $H$ used to define it, at least if $\hat{L}_0$ and $\Lambda$ are transverse.

Now suppose that $\Lambda,\Lambda'\in\mathcal{L}(\hat{L}_1)$ are both transverse to $\hat{L}_0$, with $\Lambda=(\phi_{H}^{1})^{-1}(\hat{L}_1)$ and $\Lambda'=\phi_{F}^{1}(\Lambda)$.  
Then where $G(t,m)=H(t,m)-F(t,\phi_{F}^{t}((\phi_{H}^{t})^{-1}(m)))$ we have $\phi_{G}^{t}=\phi_{H}^{t}\circ (\phi_{F}^{t})^{-1}$ and so $(\phi_{G}^{1})^{-1}(\hat{L}_1)=\Lambda'$.  Note that $F(t,m)=(H-G)(t,\phi_{G}^{t}(m))$.   By (\ref{betacont}) we have, for all $\frak{c},k$, \begin{align*} |\beta_{\frak{c},k}(\hat{L}_0,\hat{L}_1;H)-\beta_{\frak{c},k}(\hat{L}_0,\hat{L}_1;G)|&\leq \int_{0}^{1}\left(\max_{\hat{M}}(H-G)(t,\cdot)-\min_{\hat{M}}(H-G)(t,\cdot)\right)dt\\&=\int_{0}^{1}\left(\max_{\hat{M}} F(t,\cdot)-\min_{\hat{M}} F(t,\cdot)  \right)dt
\end{align*}
Since this holds for all $\frak{c},k$ it follows that $|B(\hat{L}_0,\Lambda)-B(\hat{L}_0,\Lambda')|\leq \int_{0}^{1}\left(\max_{\hat{M}} F(t,\cdot)-\min_{\hat{M}} F(t,\cdot)\right)dt$  But $F\in C^{\infty}_{c}([0,1]\times\hat{M})$ was arbitrary subject to the requirement that $\Lambda'=\phi_{F}^{1}(\Lambda)$, so this proves that $|B(\hat{L}_0,\Lambda)-B(\hat{L}_0,\Lambda')|\leq \delta(\Lambda,\Lambda')$.
 
Finally, the case in which $\Lambda$ and/or $\Lambda'$ is not transverse to $\hat{L}_0$ follows straightforwardly by a limiting argument from the transverse case, by using  $C^0$-small Hamiltonians $F,F'$ such that $\phi_{F}^{1}(\Lambda),\phi_{F'}^{1}(\Lambda')$ are transverse to $\hat{L}_0$ (and bearing in mind that, in the nontransverse case, $\beta_{\frak{c},k}$ was defined as a limit using (\ref{betacont})).
\end{proof}

\section{Reparametrized geodesic flows} \label{geodesic}

To begin the analysis leading to our main results,
we now specialize the discussion of Section \ref{gens} to the following situation.  Let $(N,g)$ be a compact Riemannian manifold, let $x_1\in N$, and let $Q\subset N$ be a compact submanifold not containing $x_1$.  The conormal bundle of $Q$ is \[ \nu^*Q=\left\{(x,p)\in T^*N|_{Q}\left| p|_{T_{x}Q}=0\right.\right\} \]  Under the identification of $T^*N$ with $TN$ given by $g$, we have an exponential map $\exp\co \nu^*Q\to N$; recall that $x_1$ is said to be a \emph{focal point} of $Q$ if $x_1$ is a critical value of this map.  In particular, by Sard's theorem, almost every $x_1\in N$ is not a focal point of $Q$.  In the case that $Q$ is a singleton $\{x_0\}$ (so $\nu^*Q=T_{x_0}^{*}N$), a focal point of $\{x_0\}$ is also known as a conjugate point of $x_0$.

The $g$-identification of $T^*N$ and $TN$ gives an inner product $\langle\cdot,\cdot\rangle$ and norm $|\cdot|$ on the fibers of $T^*N$.  Our Liouville domain will be the disk bundle $M=D^*N=\{(x,p)\in T^*N||p|\leq 1\}$, endowed with the canonical one-form $\theta_{(x,p)}(v)=p(\pi_* v)$, with boundary $S^*N=\{(x,p)||p|=1\}$.    The completion $\hat{M}$ is then just the cotangent bundle $T^*N$ with its canonical one-form $\hat{\theta}$. More specifically, where $0_N$ is the zero section of $T^*N$, the Liouville completion process identifies $T^*N\setminus 0_N$ with $(0,\infty)_r\times S^*N$, with the canonical one-form $\hat{\theta}$ identified with $r\alpha$ where \[ r(x,p)=|p|\qquad \alpha_{(x,p)}(v)=\frac{p}{|p|}(\pi_*v) \] (\emph{i.e.} $\alpha$ is the pullback of the contact form $\theta|_{S^*N}$ by the standard projection $T^*N\setminus 0_N\to S^*N$).

Our filling Lagrangians $L_0,L_1\subset M=D^*N$ will be given by the disk conormal bundle $L_0=\nu^*Q\cap D^*N$ and the disk cotangent fiber $L_1=T^{*}_{x_1}N\cap D^*N$, where $x_1\in N$ and $Q\subset N\setminus\{x_1\}$ is a compact submanifold such that $x_1$ is not a focal point of $Q$.  So in particular $\theta|_{L_0}=\theta|_{L_1}=0$ and, in the notation of Section \ref{gens}, $\hat{L}_0=\nu^*Q$ and $\hat{L}_1=T_{x_1}^{*}N$.  We will consider Hamiltonians $H\co [0,1]\times T^*N\to \R$ of the following special form: \[ H_f(t,(x,p))=f(|p|)\quad\mbox{where }f\in C^{\infty}([0,\infty)),\,supp(f)\subset [0,R),\,supp(f')\subset(0,R) \] for some  $R>0$.  

So where $\psi_t\co S^*N\to S^*N$ is the Reeb flow of $\alpha=\theta|_{S^*N}$ on $S^*N$, and where $\phi_{H_f}^{t}\co T^*N\to T^*N$ is the Hamiltonian flow of $H_f$, we see that $\phi_{H_f}^{t}$ restricts to $0_N$ as the identity (since $f'$ is supported away from $0$), while on $T^*N\setminus 0_N\cong(0,\infty)\times S^*N$, we have \begin{equation}\label{hflow} \phi_{H_f}^{t}(r,z)=(r,\psi_{f'(r)t}(z)) \end{equation}

Now the Reeb flow $\psi_t\co S^*N\to S^*N$ is well-known and easily-verified to be given by the geodesic flow: using the identification of $S^*N$ with the sphere tangent bundle given by the metric, for $z=(x,p)\in S^*N$ (so $|p|=1$) $\psi_t((x,p))$ is given by the position and 
velocity at time $t$ of the geodesic whose initial position and velocity are $x$ and $p$.  

In particular, the restriction of $\phi_{H_f}^{1}$ to $\hat{L}_0=\nu^*Q$ sends $(x,p)\in T^*N$ where $p|_{T_xQ}=0$ to the pair $(\gamma(f'(|p|)),|p|\gamma'(f'(|p|)))$,  where $\gamma$ is the geodesic with initial position and velocity given by $x_0$ and $\frac{p}{|p|}$.  In order to set up the Floer complex $CF(\nu^*Q,T_{x_1}^{*}N;H)$, we require that $\phi_{H_f}^{1}(\nu^*Q)$ be transverse to $T_{x_1}^{*}N$.  This holds if and only if, where $\pi\co T^*N\to N$ is the bundle projection, $\pi\circ\phi_{H_f}^{1}\co \nu^*Q\to N$ has $x_1$ as a regular value.  Now we have 
\[ \pi\circ\phi_{H_f}^{1}(p)= \exp\left(f'(|p|)\frac{p}{|p|}\right) \]  Since $x_1$ is assumed to not be a focal point of $Q$, $x_1$ is a regular value of $\exp\co \nu^*Q\to N$; consequently $x_1$ will be a regular value of $\pi\circ\phi_{H_f}^{1}\co \nu^*Q\to N$ provided that \begin{equation}\label{transcond} \begin{array}{cc} \mbox{If $r\in \R$ and if there is a geodesic $\gamma$ such that $\gamma'(0)\in \nu^*Q$,  $\gamma(1)=x_1$,} \\ \mbox{and $\gamma$ has length $\ell=|f'(r)|$, then $f''(r)\neq 0$}\end{array}\end{equation} (here we use the $g$-identification of $T^*N$ with $TN$ to view $\nu^*Q$ as a subset of $TN$).  

Where we denote the sphere conormal bundle by $S\nu^*Q=S^*N\cap \nu^*Q$ and the sphere cotangent fiber by $S_{x_1}^{*}N=S^*N\cap T_{x_1}^{*}N$, 
let
 \[ \mathcal{B}(Q,x_1)=\{(\tau,y)\in \R\times S\nu^*Q|\psi_{\tau}(y)\in S^{*}_{x_1}(N) \}. \]  Thus $(\tau,(x,p))\in \mathcal{B}(Q,x_1)$ if and only if the geodesic $\gamma\co \R\to N$ with $\gamma(0)=x$ and $\gamma'(0)=p$ has $\gamma(\tau)=x_1$; note that negative values of $\tau$ are permitted. So elements $(\tau,y)$ of $\mathcal{B}(Q,x_1)$ are in one-to-one correspondence with Reeb flowlines (for $\tau>0$) or negative Reeb flowlines (for $\tau<0$) of duration $|\tau|$ which begin on $S\nu^*Q$ and end on $S_{x_1}^{*}N$. The \emph{action} of the positive or negative Reeb flowline $\eta_{(\tau,y)}$ corresponding to $(\tau,y)$, defined as $\int_{\eta_{(\tau,y)}}\alpha$, is in either case just given by $\tau\in \R$. 
 
 Now $\pi_0(\mathcal{P}(\nu^*Q,T_{x_1}^{*}N))$ is in obvious bijection (via the bundle projection) with the set $\mathcal{P}_N(Q,x_1)$ of homotopy classes of paths in $N$ from $Q$ to $x_1$; this identification will be implicit in what follows.  For $\frak{c}\in \pi_0(\mathcal{P}(\nu^*Q,T_{x_1}^{*}N))$ let \[ \mathcal{B}_{\frak{c}}(Q,x_1)=\{(\tau,y)\in\mathcal{B}(Q,x_1)| [\pi\circ\eta_{(\tau,y)}]=\frak{c} \} \] where $\pi\co S^*N\to N$ is the bundle projection and as in the last paragraph $\eta_{(\tau,y)}$ is the positive or negative Reeb flowline from $S\nu^*Q$ to $S_{x_1}^{*}N$ that corresponds to the pair $(\tau,y)$.  Also
 let \[ \mathcal{B}_{\frak{c}}^{+}(Q,x_1)=\{(\tau,y)\in \mathcal{B}_{\frak{c}}(x_0,x_1)|\tau>0\} \]

The cotangent bundle $T^*M$ has an  involution $h\co T^*M\to T^*M$ defined by $h(x,p)=(x,-p)$; this obeys $h^*{\hat{\theta}}=-\hat{\theta}$ and restricts to an involution $h\co S^*M\to S^*M$ obeying $h^{*}\alpha=-\alpha$. Thus the Reeb flow $\{\psi_t\}_{t\in\R}$ obeys $h\circ\psi_t=\psi_{-t}\circ h$.   Moreover obviously $h(\hat{L}_i)=\hat{L}_i$ and $\pi\circ h=\pi$, in view of which $(\tau,y)\in \mathcal{B}_{\frak{c}}(Q,x_1)$ if and only if $(-\tau,h(y))\in \mathcal{B}_{\frak{c}}(Q,x_1)$.  Consequently, defining $h_*\co\mathcal{B}_{\frak{c}}(Q,x_1)\to \mathcal{B}_{\frak{c}}(Q,x_1)$ by $h_*(\tau,y)=(-\tau,h(y))$, we have \[ \mathcal{B}_{\frak{c}}(Q,x_1)=\mathcal{B}_{\frak{c}}^{+}(Q,x_1)\coprod h_*\left( \mathcal{B}_{\frak{c}}^{+}(Q,x_1) \right) \]

We now set up the filtered Floer complex $CF_{\frak{c}}(\hat{L}_0,\hat{L}_1;H)$.  We must first choose a basepoint $\gamma_{\frak{c}}\co [0,1]\to T^*N$ for our homotopy class $\frak{c}\in \pi_0(\mathcal{P}(\hat{L}_0,\hat{L}_1))$; we take $\gamma_{\frak{c}}$ to be an arbitrary path in the class $\frak{c}$ that is contained in the zero section $0_N$ (and so necessarily has $\gamma_{\frak{c}}(0)=(x_0,0),\gamma_{\frak{c}}(1)=(x_1,0)$ for some $x_0\in Q$).  For later use it will be convenient to assume also that $\gamma_{\frak{c}}|_{[0,1/2]}$ is constant. 

With this choice of $\gamma_{\frak{c}}$, since the $1$-form $\hat{\theta}$ vanishes on each of $0_N$, $\hat{L}_0$, $\hat{L}_1$, we see from Stokes' theorem that, for any Hamiltonian $H\co [0,1]\times T^*N\to \R$, the action functional $\mathcal{A}_H\co \frak{c}\to \R$ is given by \[ \mathcal{A}_H(\gamma)=-\int_{[0,1]}\gamma^{*}\hat{\theta}+\int_{0}^{1}H(t,\gamma(t))dt\]

Specializing to the case of our Hamiltonians $H_f(t,p)=f(|p|)$, the critical points of $\mathcal{A}_{H_f}\co \frak{c}\to \R$ may be described as follows.  Let 
\[ \mathcal{C}_{\frak{c}}(f,Q,x_1)=\{(r,y)\in (0,\infty)\times S_{x_0}^{*}N| (f'(r),y)\in \mathcal{B}_{\frak{c}}(Q,x_1)\} \]   
The critical points $\gamma\co [0,1]\to T^*N$ for $\mathcal{A}_{H_f}$ which represent the homotopy class $\frak{c}$ are all contained in $T^*N\setminus 0_N$, and with respect to our identification of $T^*N\setminus 0_N$ with $(0,\infty)\times S^*N$ are precisely the curves of the form \[ \gamma_{(r,y)}(t)=(r,\psi_{f'(r)t}(y))\qquad \left((r,y)\in \mathcal{C}_{\frak{c}}(f,Q,x_1)\right) \]

Since $\hat{\theta}=r\alpha$, we immediately see that \begin{equation}\label{actionrz} 
\mathcal{A}_{H_f}(\gamma_{(r,y)})=f(r)-rf'(r) 
\end{equation}

\subsection{Grading}\label{grading}

We now discuss the grading on the Floer complex.  The discussion is somewhat simpler when the submanifold $Q$ is a singleton $\{x_0\}$, so we consider that case first, discussing the necessary modifications for more general $Q$ later.    Thus from now through Proposition \ref{gradecomp} we assume that $\hat{L}_0=T_{x_0}^{*}N$, and as always $\hat{L}_1=T_{x_1}^{*}N$, where the points $x_0$ and $x_1$ are nonconjugate.

 First we must also choose a symplectic trivialization for $\gamma_{\frak{c}}^{*}T(T^*N)$ where $\gamma_{\frak{c}}\co [0,1]\to 0_N\subset T^*N$ is the basepoint of $\frak{c}$ chosen earlier; we choose any trivialization with the property that each of the cotangent fibers $T_{\gamma_{\frak{c}}(t)}^{*}N$ is mapped to $\R^n\times\{0\}$. Note that if $v\co [0,1]\times S^1\to T^*N$ is any map such that $v(\{i\}\times S^1)\subset \hat{L}_i$ for $i=0,1$, then we can trivialize $v^*TM$ in such a way that the tangent spaces at $v(s,t)$ to the cotangent fibers $T_{\pi(v(s,t))}^{*}N$ are all sent to $\R^n\times\{\vec{0}\}$.  Therefore the Maslov homomorphism $\mu_{\frak{c}}\co \pi_1(\frak{c},\gamma_{\frak{c}})\to\Z$ vanishes and so our grading will be by $\Z$.  

Before addressing the Maslov indices of the $\gamma_{(r,y)}$ we discuss the positive Reeb flowlines $\eta_{(\tau,y)}\co [0,\tau]\to S^*N$ (defined by $\eta_{(\tau,y)}(s)=\psi_{s}(y)$) for $(\tau,y)\in\mathcal{B}_{\frak{c}}^{+}(\{x_0\},x_1)$.  The contact distribution $\xi=\ker \alpha$ on $S^*N$ is a rank-$(2n-2)$ symplectic vector bundle over $S^*N$ (with fiberwise symplectic form given by $d\alpha$), and for all $s$ the tangent space at $\eta_{(\tau,y)}(s)$ to the sphere cotangent fiber $S_{\pi(\eta_{(\tau,y)}(s))}^{*}N$ is a Lagrangian subspace of $\xi$.  So we can symplectically trivialize $\eta_{(\tau,y)}^{*}\xi$ in a manner which sends these vertical Lagrangian subspaces to $\R^{n-1}\times\{\vec{0}\}$.  With respect to this trivialization, the path $s\mapsto (\psi_{s})_{*}T_{z}S_{x_0}^{*}N$ defines a path of Lagrangian subspaces of $\R^{2n-2}$.  This path has a Robbin-Salamon-Maslov index (with respect to $\R^{n-1}\times\{\vec{0}\}$) \cite{RS}; we denote this index by $\nu(\eta_{(\tau,y)})$.

The index $\nu(\eta_{(\tau,y)})$ is a sum of contributions corresponding to the intersections of $(\psi_{s})_{*}T_{y}S_{x_0}^{*}N$ with $T_{\psi_s(y)}S_{\pi(\psi_s(y))}^{*}N$ as $s$ varies from $0$ to $\tau$.  Recalling that $\psi_s$ is the time-$s$ geodesic flow on $S^*N$, the restriction of the linearization $(\psi_{s})_*$ to $T_{y}S_{x_0}^{*}N$ is given by the derivative of the time-$s$ version of the exponential map. So $(\psi_{s})_{*}$ maps $v\in T_{y}S_{x_0}^{*}N$ to the pair $(J_v(s),J_{v}'(s))$ where $J_v$ is the Jacobi field along the geodesic $\pi\circ \eta_{(\tau,y)}$ having $J_v(0)=0$ and $J_{v}'(0)=v$.  Consequently the intersections $(\psi_{s})_{*}T_{y}S_{x_0}^{*}N\cap T_{\psi_s(y)}S_{\pi(\psi_s(y))}^{*}N$ may be identified with the space of normal Jacobi fields along $\pi\circ\eta_{(\tau,y)}|_{[0,s]}$ which vanish at times $0$ and $s$.  The crossing form of \cite[Theorem 1.1]{RS} at time $s$ is then given by $Q((0,J'_v(s)))=d\alpha((0,J'_v(s)),(J'_v(s),J''_v(s)))=|J'_v(s)|^2$; in particular the crossing form is positive definite.

So according to the definition in \cite{RS}, the Maslov index is given by \[ \nu(\eta_{(\tau,y)})=\frac{n-1}{2}+\sum_{0<s<\tau}\dim\left( (\psi_{s})_{*}T_{y}S_{x_0}^{*}N\cap T_{\psi_s(y)}S_{\pi(\psi_s(y))}^{*}N \right)  \]  (the $\frac{n-1}{2}$ is the contribution from $s=0$; there is no contribution from $s=\tau$ since $x_0$ and $x_1$ are nonconjugate).  In other words $\nu(\eta_{(\tau,y)})-\frac{n-1}{2}$ is the number of conjugate points along the geodesic $\pi\circ\eta_{(\tau,y)}$, counted with multiplicity.  According to the Morse Index Theorem \cite[Theorem V.15.2]{Morse}, this latter quantity is precisely the Morse index of the geodesic $\pi\circ\eta_{(\tau,y)}$ (after it is reparametrized to have duration $1$), considered as a critical point of the energy functional $\eta\mapsto \int |\eta'|^2$ on paths from $x_0$ to $x_1$; thus \begin{equation} \label{numorse} \mbox{For $(\tau,y)\in\mathcal{B}_{\frak{c}}^{+}(\{x_0\},x_1)$, }\nu(\eta_{(\tau,y)})=\frac{n-1}{2}+Morse(\pi\circ\eta_{(\tau,y)}) \end{equation} where ``$Morse$'' denotes the aforementioned Morse index.

Now consider a critical point $\gamma_{(r_0,y)}$ of $\mathcal{A}_{H_f}$ with $f'(r_0)>0$.  Thus, identifying $T^*N\setminus 0_N$ with $(0,\infty)\times S^*N$ in our usual way, \[ \gamma_{(r_0,y)}(s)=(r_0,\eta_{(f'(r_0),y)}(f'(r_0)s))=(r_0,\psi_{f'(r_0)s}(y)) \]  To determine the grading of $\gamma_{(r_0,y)}$ we must first determine the Maslov index of the path $s\mapsto (\phi_{H_f}^{s})_{*}T_{(r_0,y)}T^{*}_{x_0}N$ relative to the path of cotangent fibers $s\mapsto T_{\gamma_{(r_0,y)}(s)}T^{*}_{\pi(\gamma_{(r_0,y)})}N$.  

The splitting $T^*N\setminus 0_N\cong (0,\infty)\times S^*N$ gives rise to a splitting \[ T_{(r_0,y)}T^{*}_{x_0}N\cong \R\partial_r \oplus T_y S^{*}_{x_0}N \] 
If $z\in   T_{y}S^{*}_{x_0}N $ then $(\phi_{H_f}^{s})_{*}(0,z)=(0,(\psi_{f'(r_0)s})_*z)$.  Meanwhile the element $\partial_r\in T_{(r_0,y)}T^{*}_{x_0}N$ has \[ (\phi_{H_f}^{s})_{*}\partial_r=\left(\partial_r,sf''(r_0)\eta'_{(f'(r_0),y)}(f'(r_0)s)\right).\]   In particular (under the transversality condition (\ref{transcond})) if $s>0$ then $\pi_*(\phi_{H_f}^{s})_{*}\partial_r\neq 0$, and so the intersections of $(\phi_{H_f}^{s})_{*}T_{(r_0,y)}T^{*}_{x_0}N$ with $T_{\gamma_{(r_0,y)}(s)}T^{*}_{\pi(\gamma_{(r_0,y)})}N$ are in one-to-one correspondence with the intersections of $(\psi_{f'(r_0)s})_*T_{y}S^{*}_{x_0}N$ with $T_{\psi_{f'(r_0)s}(y)}S_{\pi(\psi_{f'(r_0)s}(y))}^{*}N$; moreover the Robbin--Salamon crossing forms are identical under this correspondence.  On the other hand at $s=0$ the vector $\partial_r$ provides a single additional dimension of intersection for $\gamma_{(r_0,y)}$ in comparison to $\eta_{(f'(r_0),y)}$. We have \[ \left.\frac{d}{ds}\right|_{s=0}(\phi_{H_f}^{s})_{*}\partial_r=(0,f''(r_0)\eta'_{(f'(r_0),y)}(0)) \] with respect to the splitting $(0,\infty)\times S^*N$.  Now with the respect to the (horizontal,vertical) splitting of the sphere (co)tangent bundle we have $\eta'_{(f'(r_0),y)}(0)=(y,0)$.  Meanwhile in the (horizontal,vertical) splitting the tangent vector $\partial_r$ at $\gamma_{r_0,y}(0)$ corresponds to $(0,y)$.  Thus the crossing form evaluates on $\partial_r$ at $s=0$ as $Q(\partial_r)=\omega((0,y),(f''(r_0)y,0))=f''(r_0)|y|^2$.

Consequently when $f'(r_0)>0$ the Maslov index of $s\mapsto (\phi_{H_f}^{s})_{*}T_{(r_0,y)}T^{*}_{x_0}N$ is $\nu(\eta_{(f'(r_0),y)})+\frac{1}{2}$ if $f''(r_0)>0$ and is $\nu(\eta_{(f'(r_0),y)})-\frac{1}{2}$ if $f''(r_0)<0$.

There remains the case that $f'(r_0)<0$.  In this case, where $h\co T^*N\to T^*N$ is again given (with respect to the (horizontal,vertical) splitting) by $h(x,p)=(x,-p)$, we have (with respect to the splitting $T^*N\setminus 0_N=(0,\infty)\times S^*N$), \[ \gamma_{(r_0,y)}(s)=(r_0,\psi_{f'(r_0)s}(y))=\left(r_0,h(\psi_{-f'(r_0)s}(h(y)))\right)=(r_0,h(\eta_{(-f'(r_0),h(y))}(-f'(r_0)s))).\]  So with the exception of the term $\pm \frac{1}{2}$ coming from $f''(r_0)$ at $s=0$, the contributions to the Maslov index of $(\phi_{H_f}^{s})_{*}T_{(r_0,y)}T^{*}_{x_0}N$ correspond to the contributions to the Maslov index of $s\mapsto \psi_{-f'(r_0)s}(h(y))$ (\emph{i.e.}, to $\nu(\eta_{(-f'(r_0),h(y))})$); however the conjugation by the antisymplectic involution $h$ negates the crossing forms and so causes the respective contributions to be opposite to each other.  Meanwhile the same argument as in the previous case shows that the contribution of $\partial_r$ at $s=0$ is $\frac{1}{2}$ if $f''(r_0)>0$ and $-\frac{1}{2}$ if $f''(r_0)<0$.  

So when $f'(r_0)<0$ the Maslov index of $s\mapsto (\phi_{H_f}^{s})_*T_{(r_0,y)}T^*_{x_0}N$ is $-\nu(\eta_{-f'(r_0),y})+\frac{1}{2}$ if $f''(r_0)>0$ and is $-\nu(\eta_{-f'(r_0),y})-\frac{1}{2}$ if $f''(r_0)<0$.

In any case, projecting down to $N$ we see that $\pi\circ \gamma_{(r_0,y)}$ and $\pi\circ \eta_{(|f'(r_0)|,y)}$ represent the same geodesic (modulo positive constant time rescaling) from $x_0$ to $x_1$.  So we have, in view of (\ref{numorse}) and our grading conventions (\ref{gradeconv}):

\begin{prop} \label{gradecomp} For each element $(r,y)\in \mathcal{C}_{\frak{c}}(f,\{x_0\},x_1)$, the corresponding element $\gamma_{(r,y)}\in Crit(\mathcal{A}_{H_f})$ has Floer-theoretic grading given by \[ \mu(\gamma_{(r,y)})=\left\{\begin{array}{ll} -Morse(\pi\circ\gamma_{(r,y)}) & f'(r)>0,\,f''(r)>0 \\
1-Morse(\pi\circ\gamma_{(r,y)}) & f'(r)>0,\,f''(r)<0  \\ n-1+Morse(\pi\circ\gamma_{(r,y)}) & f'(r)<0,\,f''(r)>0 \\ n+Morse(\pi\circ\gamma_{(r,y)}) & f'(r)<0,\,f''(r)<0 \end{array}\right. \] 
\end{prop}


We now turn to the more general situation in which the submanifold $\hat{L}_0=T^{*}_{x_0}N$ is replaced by the conormal bundle $\nu^*Q$ of a smooth, compact $d$-dimensional submanifold $Q$ of $N$.  We continue to set $\hat{L}_1=T_{x_1}^{*}N$, and now require that $x_1$ is not a focal point of $Q$ (in  other words, is not a critical value of the restriction of the exponential map to $\nu^*Q$).  

A few subtleties arise in this extension because the conormal bundle $\nu^*Q$ typically has more complicated geometry than a cotangent fiber.  In particular the second fundamental form of $Q$ gives rise, for every $(x,p)\in T^*N|_Q$, to a symmetric ``shape operator'' $\Sigma_{(x,p)}\co T_x Q\to T_x Q$ defined by the property that, for $v,w\in T_{x}Q$,  \[ g(\Sigma_{(x,p)}v,w)=p\left((\nabla_v W)^{\perp}\right) \] where $W$ is a vector field on a neighborhood of $x$ in $N$ which is tangent to $Q$ and has $W(x)=w$,  $\nabla$ is the Levi-Civita connection associated to $g$ on $N$, and $\perp$ denotes projection of a vector in $TN|_Q$ to the $g$-orthogonal complement $TQ^{\perp}$ of $TQ$.  We extend $\Sigma_{(x,p)}$ to a symmetric operator $T_x N\to T_xN$ by setting it equal to zero on $T_xQ^{\perp}$.  

The Levi-Civita connection $\nabla$ (together with the $g$-identification of $T^*N$ with $TN$) gives us a splitting $T(T^*N)=T^{vt}\oplus T^{hor}$ where $T^{vt}$ is the vertical tangent space of the projection $T^*N\to N$ and where, for a section $s\co N\to T^*N$ with $s(x)=(x,p)$ and for $v\in T_x N$, we have $s_*v\in T_{(x,p)}^{hor}$ if and only if, viewing $s$ as a vector field, $(\nabla_v s)(x)=0$.    For $(x,p)\in T^*N$, we may identify both   $T^{hor}_{(x,p)}$ and $T^{vt}_{(x,p)}$ with $T_{x}N$. Indeed, the projection-induced map $\pi_*\co T_{(x,p)}T^*N\to T_x N$ restricts as an isomorphism $T^{hor}_{(x,p)}\cong T_x N$; we will denote the inverse of this isomorphism by $v\mapsto v^{\#}$ (so $v^{\#}$ is the horizontal lift of $v$ in the standard sense).  Meanwhile $T^{vt}_{(x,p)}$ is identified with $T_{x}^{*}N$ by the identification of the tangent space to a vector space with the vector space, and $T_{x}^{*}N$ is identified with $T_{x} N$ by the metric $g$.  For $v\in T_x N$ we will denote the corresponding element of $T^{vt}_{(x,p)}$ by $v^{\flat}$.

 Now for $(x,p)\in \nu^*Q$, define \[ \tilde{T}^{hor}_{(x,p)}=\{v^{\#}-(\Sigma_{(x,p)}v)^{\flat}|v\in T_x N\}.\]  We then have a decomposition of bundles \[ T(T^*N)|_{\nu^*Q}=T^{vt}\oplus \tilde{T}^{hor},\] and the fact that the $\Sigma_{(x,p)}$ are symmetric translates to the statement that $\tilde{T}^{hor}$ is (like both $T^{hor}$ and $T^{vt}$) a \emph{Lagrangian} subbundle of $T(T^*N)|_{\nu^*Q}$ with respect to the standard symplectic form $d\hat{\theta}$. For $v\in T_x N$ and $(x,p)\in \nu^*Q$ define $v^{\widetilde{\#}}=v^{\#}-(\Sigma_{(x,p)}v)^{\flat}$; in other words, $v^{\widetilde{\#}}$ is the unique element of $\tilde{T}_{(x,p)}^{hor}$ that projects to $v$.

A straightforward calculation shows that the tangent space to the conormal bundle $\nu^*Q$ is given as follows: \begin{equation}\label{nusplit} T\nu^*Q=\{v^{\flat}|v\in TQ^{\perp}\}\oplus \{v^{\widetilde{\#}}|v\in TQ\}=:(TQ^{\perp})^{\flat}\oplus (TQ)^{\widetilde{\#}}. \end{equation}

\begin{prop} \label{At} There is a smooth family of Lagrangian subbundles $A_t\leq T(T^*N)|_{\nu^*Q}$ ($0\leq t\leq 1$) such that: \begin{itemize}
\item $A_0=T\nu^*Q$
\item $A_1=T^{vt}$
\item For each $(x,p)\in \nu^*Q$, the Robbin--Salamon--Maslov index of the path $\{(A_t)_{(x,p)}\}_{0\leq t\leq 1}$ with respect to the constant path $T_{(x,p)}\nu^*Q$ is given by \[ \mu_{RS}\left((A_t)_{(x,p)},T_{(x,p)}\nu^*Q\right)=\frac{d}{2}\] where $d=\dim Q$.
\end{itemize}
\end{prop}  

\begin{proof} Define an endomorphism $J\co T(T^*N)|_{\nu^*Q}\to T(T^*N)|_{\nu^*Q}$ by $Jv^{\flat}=v^{\tilde{\#}}$ and $Jv^{\tilde{\#}}=-v^{\flat}$.  Thus $J$ is a $d\hat{\theta}$-compatible almost complex structure on $T(T^*N)|_{\nu^*Q}$, mapping $T^{vt}$ to $\tilde{T}^{hor}$ and vice versa. 

For any $s\in \R$, define an endomorphism $e^{sJ}\co T(T^*N)|_{\nu^*Q}\to T(T^*N)|_{\nu^*Q}$ by $e^{sJ}w=(\cos s)w+(\sin s)Jw$.  Then let \[ A_t=(TQ^{\perp})^{\flat}\oplus e^{\frac{\pi t}{2}J}\left(TQ^{\widetilde{\#}}\right) \]  The $A_t$ are easily seen to be Lagrangian, and we have $A_0=T\nu^*Q$ by (\ref{nusplit}), while $A_1=(TQ^{\perp})^{\flat}\oplus TQ^{\flat}=T^{vt}$.  Meanwhile the product axiom for the Robbin--Salamon--Maslov index shows that $\mu_{RS}((A_t)_{(x,p)},T_{(x,p)}\nu^*Q)$ is equal to the Robbin--Salamon--Maslov index of the path $t\mapsto e^{\frac{\pi t}{2}J}\left(TQ^{\widetilde{\#}}_{(x,p)}\right)$ with respect to $TQ^{\widetilde{\#}}_{(x,p)}$ in the $2d$-dimensional symplectic vector space
$TQ^{\flat}_{(x,p)} \oplus TQ^{\widetilde{\#}}_{(x,p)}$.  The only intersections relevant to this index are those at $t=0$, and at $t=0$ the crossing form is easily seen to be positive definite, so that the index is indeed $\frac{d}{2}$. 
\end{proof}

With this preparation we can discuss the gradings of the Floer complexes $CF_{\frak{c}}(\nu^*Q,T_{x_1}^{*}N;H_f)$.  Recall that just before the start of Section \ref{grading} we chose a representative $\gamma_{\frak{c}}$ of the homotopy class $\frak{c}\in \pi_0(\mathcal{P}(\nu^*Q,T_{x_1}^{*}N))$, such that $\gamma_{\frak{c}}(t)\in 0_N$ for all $t$ and $\gamma_{\frak{c}}|_{[0,1/2]}$ is constant.  For the grading we must also fix a suitable trivialization $\tau_{\frak{c}}$ of $\gamma_{\frak{c}}^{*}T(T^*N)$.    We choose this symplectic trivialization $\tau_{\frak{c}}\co \gamma_{\frak{c}}^{*}T(T^*N)\to [0,1]\times\R^{2n}$ in such a way that \[ \tau_{\frak{c}}^{-1}(\{t\}\times \R^n\times\{\vec{0}\})=\left\{\begin{array}{ll}(A_{2t})_{\gamma_{\frak{c}}(0)} & 0\leq t\leq\frac{1}{2} \\ T^{vt}_{\gamma_{\frak{c}(t)}} & \frac{1}{2}\leq t\leq 1\end{array}\right. \]

Any element of $\pi_1(\frak{c},\gamma_{\frak{c}})$ can be represented by a map $v\co S^1\times [0,1]\to T^*N$ such that $v(\theta,t)\in \nu^*Q$ for $\theta\in S^1$ and $0\leq t\leq 1/2$, while $v(\theta,1)\in T_{x_1}^{*}N$.  Then $v^*T(T^*N)$ has a Lagrangian subbundle $\Lambda$ whose fiber over $(\theta,t)$ is given by $(A_{2t})_{v(\theta,t)}$ for $0\leq t\leq 1/2$ and by $T^{vt}_{v(\theta,t)}$ for $1/2\leq t\leq 1$.  Since $\Lambda$ restricts to $S^1\times\{0\}$ as the pullback of $T\nu^*Q$ and to $S^1\times\{1\}$ as the pullback of $T(T_{x_1}^{*}N)$, it follows that the Maslov homomorphism $\mu_{\frak{c}}\co \pi_1(\frak{c},\gamma_{\frak{c}})\to\Z$ vanishes on $[v]$.  Thus the grading of the Floer complexes $CF_{\frak{c}}(\nu^*Q,T_{x_1}^{*}N;H)$ is by $\Z$.

Where  $S\nu^*Q=S^*N\cap \nu^*Q$ is the sphere conormal bundle of $Q$ and where $\psi_t\co S^*N\to S^*N$ is the Reeb flow, recall the sets \[ \mathcal{B}_{\frak{c}}(Q,x_1)=\{(\tau,y)\in \R\times S\nu^*Q|\psi_{\tau}(y)\in S^*_{x_1}N,\,[s\mapsto \psi_{s\tau}(y)]\in\frak{c}\}\] and \[ \mathcal{C}_{\frak{c}}(f,Q,x_1)=\{(r,y)\in  (0,\infty)\times S\nu^*Q|(f'(r),y)\in \mathcal{B}_{\frak{c}}(Q,x_1)\}.\]  The critical points of $\mathcal{A}_{H_f}\co\frak{c}\to \R$ are the paths \[ \gamma_{(r,y)}(t)=r\psi_{f'(r)t}(y) \qquad (r,y)\in \mathcal{C}_{\frak{c}}(f,Q,x_1)\] 

It follows from our choice of $\tau_{\frak{c}}$ together with the catenation and homotopy properties of the Robbin--Salamon--Maslov index $\mu_{RS}$ for pairs of Lagrangian arcs (\cite[Section 3]{RS}) that \begin{equation}\label{combine} \mu(\gamma_{(r,y)})=\frac{n}{2}-\left(\mu_{RS}(T_{ry}\nu^*Q,t\mapsto A_t)+\mu_{RS}(t\mapsto (\phi_{H_f}^{t})_*T_{ry}\nu^*Q,T^{vt}_{\phi_{H_f}^{t}(ry)})\right). \end{equation}  By Proposition \ref{At} (and the fact that  $\mu_{RS}$ is antisymmetric in its two arguments) we have  $\mu_{RS}(T_{ry}\nu^*Q,t\mapsto A_t)=-\frac{d}{2}$.  Meanwhile the computation of $\mu_{RS}(t\mapsto (\phi_{H_f}^{t})_*T_{ry}\nu^*Q,T^{vt}_{\phi_{H_f}^{t}(ry)})$ essentially repeats what was done in the special case that $Q=\{x_0\}$.  Intersections of  $(\phi_{H_f}^{t})_*T_{ry}\nu^*Q$ with  $T^{vt}_{\phi_{H_f}^{t}(ry)}$ where $t>0$ correspond to focal points of $Q$ along the geodesic $\pi\circ\gamma_{(r,y)}$; these contribute positively to $\mu_{RS}$ when $f'(r)>0$ and negatively to $\mu_{RS}$ when $f'(r)<0$, in each case according to their multiplicity.  Now \cite[Theorem V.15.2]{Morse} shows that the sum of the multiplicities of the focal points along $\pi\circ\gamma_{(r,y)}$ is equal to the Morse index $Morse(\pi\circ\gamma_{(r,y)})$  of the geodesic $\pi\circ\gamma_{(r,y)}$, considered as a critical point of the energy functional on paths from $Q$ to $x_1$.  Thus $\mu_{RS}(t\mapsto (\phi_{H_f}^{t})_*T_{ry}\nu^*Q,T^{vt}_{\phi_{H_f}^{t}(ry)})$ is equal to $sign(f'(r))\cdot Morse(\pi\circ\gamma_{(r,y)})$ plus one-half of the signature of the crossing form at $t=0$.  This latter signature is computed just as in the previously-considered case that $Q=\{x_0\}$; the only difference is that the dimension of $T_{ry}\nu^*Q\cap T^{vt}_{ry}$ is $n-d$ rather than $n$.  The crossing form evaluates with the same sign as $f''(r)$ on the radial tangent vector, and with the same sign as $f'(r)$ on each nonzero element of the orthogonal complement of the radial tangent vector.  So we obtain \[ \mu_{RS}(t\mapsto (\phi_{H_f}^{t})_*T_{ry}\nu^*Q,T^{vt}_{\phi_{H_f}^{t}(ry)})=\left\{\begin{array}{ll} \frac{n-d}{2}+Morse(\pi\circ\gamma_{(r,y)}) & f'(r)>0,f''(r)>0 \\ 
\frac{n-d}{2}-1+Morse(\pi\circ\gamma_{(r,y)}) & f'(r)>0,f''(r)<0 \\ 
\frac{d-n}{2}+1-Morse(\pi\circ\gamma_{(r,y)}) & f'(r)<0,f''(r)>0 \\
\frac{d-n}{2}-Morse(\pi\circ\gamma_{(r,y)}) & f'(r)<0,f''(r)<0
\end{array}\right.\]
So by (\ref{combine}) we may generalize Proposition \ref{gradecomp} as follows: 

\begin{prop}\label{gradecompQ} For each element $(r,y)\in \mathcal{C}_{\frak{c}}(f,Q,x_1)$, the corresponding element $\gamma_{(r,y)}\in Crit(\mathcal{A}_{H_f})$ has Floer-theoretic grading given by \[ \mu(\gamma_{(r,y)})=\left\{\begin{array}{ll}
d-Morse(\pi\circ\gamma_{(r,y)}) & f'(r)>0,f''(r)>0  \\
d+1-Morse(\pi\circ\gamma_{(r,y)}) & f'(r)>0,f''(r)<0 \\
n-1+Morse(\pi\circ\gamma_{(r,y)}) & f'(r)<0,f''(r)>0 \\
n+Morse(\pi\circ\gamma_{(r,y)}) & f'(r)<0,f''(r)<0
\end{array}\right.\]
\end{prop}

\subsection{An assumption on the geodesics of $N$}\label{assumption}

Fix now a compact  $d$-dimensional submanifold $Q\subset N$ and a point $x_1\in N$ which neither lies in $Q$ nor is a focal point of $Q$.  Let \[ \mathcal{G}(Q,x_1)=\left\{\gamma\co [0,1]\to N\left|\begin{array}{cc}\gamma(0)\in Q,\,\gamma'(0)\in T_{\gamma(0)}Q^{\perp},\,\gamma(1)=x_1\\ \gamma\mbox{ is a geodesic }   \end{array}\right.\right\} \] Regarding any $\gamma\in \mathcal{G}(Q,x_1)$ as a path in the zero section $0_N\subset N$, $\gamma$ represents a class $[\gamma]\in \pi_0(\mathcal{P}(\nu^*Q,T_{x_1}^{*}N))$.  Moreover $\gamma$ is a critical point of the energy functional $\eta\mapsto\int|\eta'|^2$ on the space $\mathcal{P}_N(Q,x_1)$ of paths from $P$ to $x_1$, and this energy functional is Morse since we assume that $x_1$ is not a focal point of $Q$.  Accordingly any $\gamma\in \mathcal{G}(Q,x_1)$ has a Morse index, which as before we denote by $Morse(\gamma)$.  By \cite[Theorem V.15.2]{Morse} this Morse index is equal to the number of focal points along $\gamma$, counted with multiplicity.  For $\frak{c}\in \pi_0(\mathcal{P}(\nu^*Q,T_{x_1}^{*}N))$ and $l\in \Z$ we define \[ \mathcal{G}_{\frak{c},l}(Q,x_1)=\{\gamma\in\mathcal{G}(Q,x_1)|[\gamma]=\frak{c},\,Morse(\gamma)=l\}\]

For our main results we now make the following assumption on the behavior of geodesics between $Q$ and $x_1$:

\begin{assm}\label{assm}  There is a homotopy class $\frak{c}\in \pi_0(\mathcal{P}(\nu^*Q,T_{x_1}^{*}N))$ and an integer $k$ such that:
\begin{itemize} \item[(i)] $\mathcal{G}_{\frak{c},k}(Q,x_1)\neq\varnothing$.
\item[(ii)] $\mathcal{G}_{\frak{c},k}(Q,x_1)\cup \mathcal{G}_{\frak{c},k+2}(Q,x_1)$ is a finite set.
\item[(iii)] $\mathcal{G}_{\frak{c},k-1}(Q,x_1)\cup\mathcal{G}_{\frak{c},k+1}(Q,x_1)=\varnothing$
\item[(iv)] Either $n-d\neq 2$ or $k\neq 0$.
\end{itemize}
\end{assm}

\begin{remark}\label{kn1}  Since there are minimizing geodesics in every homotopy class of paths from $Q$ to $x_1$, and since such geodesics are always perpendicular to $Q$, it always holds that $\mathcal{G}_{\frak{c},0}(Q,x_1)\neq\varnothing$.  In particular we can never have $k=1$ in Assumption \ref{assm}, as this would violate item (iii).
\end{remark}

\begin{remark}  Assumption \ref{assm} is most often satisfied with $k=0$; in fact, I do not know of an example  in which the condition holds for some nonzero $k$ but does not hold for $k=0$.  When $k=0$,  (i) always holds (as mentioned in Remark \ref{kn1}), while (iii) requires only that $\mathcal{G}_{\frak{c},1}(Q,x_1)=\varnothing$.  
\end{remark}

\begin{remark}\label{pi2}  Since the Morse complex generated by geodesics from $Q$ to $x_1$ has homology equal to that of  space $\mathcal{P}_N(Q,x_1)$ of paths in $N$ from $Q$ to $x_1$ (cf. the proof of Proposition \ref{nonpos} below), in order for Assumption \ref{assm} to hold it is necessary for $\mathcal{P}(Q,x_1)$ to have a component with nontrivial homology in degree $k$ but trivial homology in degrees $k-1$ and $k+1$.  In the case that $Q$ is a single point, $\mathcal{P}_N(Q,x_1)$ is homeomorphic to the based loop space $\Omega N$. In particular, if $Q$ is a single point, then in order for Assumption \ref{assm} to hold with $k=0$ it is necessary that $\pi_2(N)=0$ (since there are isomorphisms $\pi_2(N)\cong\pi_1(\Omega N)\cong H_1(\Omega N;\Z)$).
\end{remark}

\begin{remark} \label{ricci}
If the compact Riemannian manifold $(N,g)$ has positive Ricci tensor everywhere, then Assumption \ref{assm} (ii) automatically holds.  Indeed in this case \cite[Theorem 19.6]{Mil} gives an upper bound on the length of any geodesic with index $k$ or $k+2$, and since $x_1$ is not a focal point of $Q$ there can be only finitely many elements of $\mathcal{G}(Q,x_1)$ whose lengths obey this upper bound.
\end{remark}

Assumption \ref{assm} is frequently satisfied in nonpositive curvature:

\begin{prop}\label{nonpos} Let $(N,g)$ be a compact connected Riemannian $n$-manifold and suppose that there is $c\geq 0$ such that all sectional curvatures of $(N,g)$ are bounded above by $-c$.  Let $d<n$ with $d\neq n-2$ and suppose that $Q\subset N$ is a compact $d$-dimensional submanifold such that for every $x\in Q$ and every $p\in T^{*}_{x}N$ with $p|_{TQ}=0$ and $|p|=1$ the shape operator $\Sigma_{(x,p)}\co T_xQ\to T_x Q$ has all of its eigenvalues bounded above by $\sqrt{c}$.  Then for any $x_1\in N\setminus Q$  and any $\frak{c}\in \pi_0(\mathcal{P}(\nu^*Q,T_{x_1}^{*}N))$, Assumption \ref{assm} holds with respect to the data $N,g,Q,x_1,\frak{c}$, and $k=0$.
\end{prop}

\begin{proof} By \cite[Theorem 4.1 and Corollary 4.2]{W}\footnote{Note that \cite{W} uses an opposite sign convention to ours for the second fundamental form.  Also, in the remark above \cite[Corollary  4.2]{W}, various references to $c^{1/2}$ should be replaced by $|c|^{1/2}$ throughout, as can be seen by considering the appropriate constant-curvature examples when $c<0$.}, the assumption on the curvature and on the eigenvalues of $\Sigma_{(x,p)}$ imply that there are no focal points along any geodesic in $\mathcal{G}(Q,x_1)$ (in particular, $x_1$ is not a focal point of $Q$).  Consequently for any $\frak{c}$ we have $\mathcal{G}_{\frak{c},l}(Q,x_1)=\varnothing$ for $l\geq 1$.  Since we have assumed that $d\neq n-2$ and since $\mathcal{G}_{\frak{c},0}(Q,x_1)$ is nonempty, this proves all parts of Assumption \ref{assm} except for the statement that $\mathcal{G}_{\frak{c},0}(Q,x_1)$ is finite.  In fact we will show that $\mathcal{G}_{\frak{c},0}(Q,x_1)$ has just one element.\footnote{In the case that $Q$ is totally geodesic, \emph{i.e.} that each $\Sigma_{(x,p)}=0$, one can show very directly that $\mathcal{G}_{\frak{c}}(Q,x_1):=\cup_l\mathcal{G}_{\frak{c},l}(Q,x_1)$ has just one element (without appealing to \cite{W} or to the Morse index theorem) simply by observing that if $\gamma_1,\gamma_2$ were distinct elements of $\mathcal{G}_{\frak{c}}(Q,x_1)$ then combining $\gamma_1$ and $\gamma_2$ with a suitable geodesic in $Q$ from $\gamma_1(0)$ to $\gamma_2(0)$ would yield a geodesic triangle in $N$ having two right angles, which is impossible since $N$ has nonpositive curvature.}

The projection $\pi\co T^*N\to N$ induces a bijection $\pi_*$ between $\pi_0(\mathcal{P}(\nu^*Q,T_{x_1}^{*}N))$ and the space $\mathcal{P}_N(Q,x_1)$ of paths from $Q$ to $x_1$, and just as in \cite[Theorem 17.3]{Mil} the path component $\pi_*\frak{c}$ of $\mathcal{P}_N(Q,x_1)$ has the homotopy type of a cell complex with one $l$-cell for every element of $\mathcal{G}_{\frak{c},l}(Q,x_1)$ (see \cite[Section 3]{K} for details on the extension of \cite[Theorem 17.3]{Mil} to the case where the left endpoints of the geodesics are replaced by a submanifold).  
But then since we have established that $\mathcal{G}_{\frak{c},l}(Q,x_1)=\varnothing$ for $l\geq 1$, there are no cells of dimension greater than one in this cell complex, and so since $\pi_*\frak{c}$ is path-connected there can only be one $0$-cell.  Thus indeed $\mathcal{G}_{\frak{c},0}(Q,x_1)$ has just one element.
\end{proof}

Meanwhile, here are some positive curvature cases in which Assumption \ref{assm} holds:

\begin{prop} \label{groupsphere} Let $(N,g)$ be either a compact semisimple Lie group with a bi-invariant metric, or a sphere $S^n$ where $n\geq 3$ with its standard metric.  Then Assumption \ref{assm} holds with $k=0$ if we choose $Q$ to consist of a single point $x_0$ which is not conjugate to $x_1$.
\end{prop}

\begin{proof} $(N,g)$ has positive Ricci curvature (as is obvious in the case of $S^n$, and follows from \cite[p. 115]{Mil} in the Lie group case since a semisimple Lie algebra has trivial center), so Assumption \ref{assm}(ii) holds by Remark \ref{ricci}.  Of course Assumption \ref{assm}(i) holds since we are taking $k=0$, and Assumption \ref{assm}(iv) holds since we have  excluded $S^2$ from the hypotheses and since there are no  two-dimensional compact semisimple Lie groups.  As for Assumption \ref{assm}(iii),  \cite[Theorem 21.7]{Mil} is proven by showing that the geodesics connecting two nonconjugate points on a compact Lie group always have even Morse index, so evidently $\mathcal{G}_{\frak{c},1}(\{x_0\},x_1)=\varnothing$ in the Lie group case.  Similarly, in the case of $S^n$ all geodesics have index divisible by $n-1$, so that $\mathcal{G}_{\frak{c},1}(\{x_0\},x_1)=\varnothing$ (again using that $n\geq 3$).
\end{proof}

\begin{remark}\label{curvop}  There is a somewhat broader class of positively-curved symmetric spaces, which includes those from Proposition \ref{groupsphere}, to which Assumption \ref{assm} applies.  Consider a compact-type symmetric space $(N,g)$, given as a Riemannian quotient $N=G/H$ where the compact Lie group $G$ is the identity component of the isometry group of $N$ and the isotropy group $H$ is a union of components of the fixed locus of an involution $\sigma$ of $G$.  
Decompose the Lie algebra $\frak{g}$ of $G$ as $\frak{g}=\frak{h}\oplus\frak{p}$ where $\frak{h}$ is the Lie algebra of $H$ and $\frak{p}$ is the $(-1)$-eigenspace of the linearization of $\sigma$.  Thus the tangent space $T_{x_0}N$ at a suitable basepoint $x_0\in N$ is naturally identified with $\frak{p}$, and we have \[ [\frak{h},\frak{h}]\subset \frak{h},\quad [\frak{p},\frak{h}]\subset \frak{p},\quad [\frak{p},\frak{p}]\subset\frak{h}. \]  Under the identification $T_{x_0}N\cong \frak{p}$ the curvature tensor of $(N,g)$ at $x_0$ is given by \[ R(X,Y)Z=-[[X,Y],Z]\quad (X,Y,Z\in\frak{p}),\]  and moreover the curvature tensor is parallel (see, \emph{e.g.}, \cite[Chapter IV]{He}).  So if $\gamma$ is a geodesic with $\gamma'(0)=X\in T_{x_0}N$, then $t$ is a conjugate time for $\gamma$ if and only if $t=\frac{2\pi j}{\sqrt{\lambda}}$ where $j\in\Z_+$ and $\lambda$ is a positive eigenvalue of the operator $Y\mapsto [[X,Y],X]$ on $\frak{p}$, with the multiplicity of the conjugate time equal to the multiplicity of the corresponding eigenvalue $\lambda$. 

In particular, it follows that $N=G/H$ obeys Assumption \ref{assm} with $k=0$,  with $x_0$ equal to any point not conjugate to $x_1$, and with any homotopy class $\frak{c}$, provided that the following condition holds: \begin{equation}\label{symcond}\mbox{For all $X\in\frak{p}$, the largest eigenvalue of $(Y\in\frak{p})\mapsto [[X,Y],X]$ has multiplicity greater than one} \end{equation}  (Indeed, the argument just given shows that, $\mathcal{G}_{\frak{c},1}(\{x_0\},x_1)=\varnothing$, and since $N$ is assumed to be of compact type it has positive Ricci tensor, so Assumption \ref{assm} (ii) holds by Remark \ref{ricci}.  Also, since $[[X,X],X]=0$ the condition (\ref{symcond}) forces $\dim N\geq 3$, so 
Assumption \ref{assm} (iv) holds.)

To verify (\ref{symcond}) in a given example, it is convenient to note that it only needs to be checked for all $X$ belonging to a fixed maximal abelian subspace $\frak{a}\subset \frak{p}$, since \cite[Lemma V.6.3]{He} shows that there is an isometry of $N$ whose derivative maps any given element of $\frak{p}$ into $\frak{a}$.    With this said, we leave it to the reader to check that (\ref{symcond}) and hence Assumption \ref{assm} holds for the following classes of symmetric spaces (for all but the first and last of these, the subspace $\frak{p}\subset \frak{g}$ and a convenient choice of the maximal abelian subspace $\frak{a}\subset\frak{p}$ are indicated in \cite[Section X.2.3]{He}; for the case of rank-one symmetric spaces, including $\mathbb{O}P^2$, see also \cite[p. 171]{He65}---the relevant eigenvalue multiplicity is denoted there by $q$):

\begin{itemize} \item $G_0\cong \frac{G_0\times G_0}{\Delta}$ where $G_0$ is a compact semisimple Lie group and $\Delta\leq G_0\times G_0$ is the diagonal.
\item The spheres $S^n=\frac{SO(n+1)}{SO(n)}$ and real projective spaces $\R P^n=\frac{SO(n+1)}{S(O(1)\times O(n))}$, where $n\geq 3$.
\item The quaternionic Grassmannians $Gr_m(\mathbb{H}^n)=\frac{Sp(n)}{Sp(m)\times Sp(n-m)}$ for $1\leq m\leq n-1$.
\item The space $\frac{U(2n)}{Sp(n)}$ of unitary quaternionic structures on $\mathbb{C}^{2n}$.
\item The Cayley projective plane $\mathbb{O}P^2=\frac{F_4}{Spin(9)}$.
\end{itemize}  

On the other hand, the members of those infinite families in Cartan's classification of irreducible compact-type symmetric spaces that are not listed above all have nonzero $\pi_2$ and so, by Remark \ref{pi2}, cannot satisfy Assumption \ref{assm} with $k=0$ and $Q$ equal to a singleton.

Note also that if $(N,g)$ is instead just \emph{locally} isometric to a compact-type symmetric space $G/H$ obeying (\ref{symcond}) then the same argument applies to show that Assumption \ref{assm} holds for $(N,g)$.  (For instance, $N$ could be a quotient of $G/H$ by a discrete group of isometries.)
\end{remark}

\begin{remark}\label{product} If the  manifold $(N,g)$ (together with auxiliary data $Q,x_1,\frak{c}$)   satisfies  Assumption \ref{assm}, and if $(N',g')$ is an arbitrary compact Riemannian manifold, then the product $(N\times N',g\times g')$ also satisfies Assumption \ref{assm}  (using auxiliary data $Q\times N', (x_1,x_2),\frak{c}\times \frak{c}_0$ where $x_2\in N'$ is chosen arbitrarily and $\frak{c}_0$ is the homotopy class of the constant path at $x_2$).  Indeed, it is easy to see that the geodesics in the product $N\times N'$ from $Q\times N'$ to $(x_1,x_2)$ which represent the class $\frak{c}\times \frak{c}_0$ and are initially orthogonal to $Q\times N'$ are precisely maps of the form $\tilde{\eta}\co t\mapsto (\eta(t),x_2)$ where $\eta$ is a geodesic in $N$ from $Q$ to $x_1$ which is initially orthogonal to $Q$.  Moreover, the homotopy classes $\frak{c}\times \frak{c}_0$ in $\pi_0(\mathcal{P}(Q\times N',(x_1,x_2)))$ remain distinct as $\frak{c}$ varies through $\pi_0(\mathcal{P}(Q,x_1))$, and the conjugate times and multiplicities for a geodesic $\eta$ from $Q$ to $x_1$ are the same as those for its corresponding geodesic $\tilde{\eta}$, so that the Morse indices coincide under the correspondence $\eta\leftrightarrow \tilde{\eta}$. 

   In particular we obtain in this way Riemannian manifolds obeying Assumption \ref{assm} that have indefinite Ricci tensor, to go along with our previous positively- and negatively-curved examples.
\end{remark}

\section{Estimating the boundary depth} \label{compute}

Let $(N,g)$ be a compact connected Riemannian $n$-manifold satisfying Assumption \ref{assm} with respect to a compact $d$-submanifold $Q$ and a point $x_1$, a nonnegative integer $k$, and a homotopy class $\frak{c}$.   We consider the Floer complexes $CF_{\frak{c}}(\nu^*Q,T_{x_1}^{*}N;H_{f_{\frak{a}}})$ for certain functions $f_{\vec{a}}\co [0,\infty)\to \R$ associated to vectors $\vec{a}\in\R^{\infty}$ that will be described presently.

First let us fix an arbitrary $R>0$ and a smooth function $h\co \R\to[0,1]$ with the following properties:\begin{itemize}
\item $supp(h)=[\delta,R-\delta]$ for some real number $\delta$ with $0<\delta<\frac{R}{4}$.
\item The only local extremum of $h|_{(\delta,R-\delta)}$ is a maximum, at $h(R/2)=1$.
\item $h''(s)<0$ if and only if $s\in (R/4,3 R/4)$.
\item $h''(s)>0$ if and only if $s\in (\delta,R/4)\cup (3R/4,R-\delta)$.\end{itemize}

Now for $\vec{a}\in \R^{\infty}$ define \[ f_{\vec{a}}(s)=\sum_{i=0}^{\infty}a_if(2^{i+1}s-R) \]  So the restriction of $f_{\vec{a}}$ to each interval $[2^{-(i+1)}R,2^{-i}R]$ is a copy of $h$ which has been rescaled horizontally to have support within the interval and which has been rescaled vertically by the factor $a_i$.

\begin{prop}\label{k-C}  Under Assumption \ref{assm},  there is a constant $C_0>0$ depending on the Riemannian metric $g$ and the function $h$ but not on $\vec{a}$ with the properties that, whenever $\nu^*Q\pitchfork (\phi_{H_{f_{\vec{a}}}}^{1})^{-1}(T^{*}_{x_1}N)$: \begin{itemize}\item[(i)] For every $(r,y)\in \mathcal{C}_{\frak{c}}(f_{\vec{a}},Q,x_1)$ such that $\mu(\gamma_{(r,y)})\in \{d-k-1,d-k+1\}$ we have $\mathcal{A}_{H_{f_{\vec{a}}}}(\gamma_{(r,y)})\geq -C_0$. \item[(ii)]  For every $(r,y)\in \mathcal{C}_{\frak{c}}(f_{\vec{a}},Q,x_1)$ such that $\mu(\gamma_{(r,y)})\in \{n+k-1,n+k+1\}$ we have $\mathcal{A}_{H_{f_{\vec{a}}}}(\gamma_{(r,y)})\leq C_0$.
\end{itemize}
\end{prop}

\begin{proof} We first prove (i). Consulting Proposition \ref{gradecompQ}, we see that any such $(r,y)$ would correspond to a geodesic $\pi\circ\gamma_{(r,y)}$ whose Morse index is one of the following: \begin{itemize} \item[(A)] $k\pm 1$, if $f_{\vec{a}}'(r)>0,f_{\vec{a}}''(r)>0$
\item[(B)] $k$ or $k+2$, if $f_{\vec{a}}'(r)>0,f_{\vec{a}}''(r)<0$
\item[(C)] $-k-(n-d)$ or $2-k-(n-d)$, if $f_{\vec{a}}'(r)<0,f_{\vec{a}}''(r)>0$
\item[(D)] $-(n-d)-k\pm 1$, if $f_{\vec{a}}'(r)<0,f_{\vec{a}}''(r)<0$ \end{itemize}
Now (A) above is forbidden by  Assumption \ref{assm}(iii).   (C) is also forbidden: for the first subcase of (C) just note that there are no geodesics of negative Morse index and we have $k\geq 0$ and $n-d\geq 1$ (as $Q$ is a compact $d$-submanifold of the connected $n$-manifold $N$ which does not contain $x_1$ and so has positive codimension), so the Morse index cannot be $-k-(n-d)$.  If the Morse index were $2-k-(n-d)$,  then in view of Assumption \ref{assm}(iv) and the facts that $n-d\geq 1$ and (by Remark \ref{kn1}) $k\neq 1$, it would have to hold  that $k=0$ and $n-d=1$.  But this is also impossible, since then $Morse(\pi\circ\gamma_{(r,y)})$ would be $1$ whereas by Assumption \ref{assm}(iii) $\mathcal{G}_{\frak{c},1}(Q,x_1)=\varnothing$.  So one of (B) or (D) applies, and in particular our element 
$(r,y)\in \mathcal{C}_{\frak{c}}(f_{\vec{a}},Q,x_1)$ has $f_{\vec{a}}''(r)<0$.  Also, the only way that (D) could hold is if $k=0$, $n-d=1$, and $\pi\circ\gamma_{(r,y)}$ has index zero.  So in any case $\pi\circ\gamma_{(r,y)}$ has index either $k$ or $k+2$.

In view of Assumption \ref{assm}(ii) there is an upper bound, say $L$, on the lengths of index-$k$ or $k+2$ geodesics in class $\frak{c}$ from $Q$ to $x_1$.  Our element $(r,y)\in  \mathcal{C}_{\frak{c}}(f_{\vec{a}},Q,x_1)$ has the property that $\pm f_{\vec{a}}'(r)$ is equal to the length of such a geodesic.  Now $\mathcal{A}_{H_{f_{\vec{a}}}}(\gamma_{(r,y)})=f_{\vec{a}}(r)-rf_{\vec{a}}'(r)$, and we have \[ |rf_{\vec{a}}'(r)|\leq RL \]  So if $f_{\vec{a}}(r)\geq 0$ obviously $\mathcal{A}_{H_{f_{\vec{a}}}}(\gamma_{(r,y)})\geq -RL$.  So assume $f_{\vec{a}}(r)<0$.  Choosing $j$ such that $r\in [2^{-(j+1)}R,2^{-j}r]$, we have $f_{\vec{a}}(r)=a_jh(2^{j+1}r-R)$ (so $a_j<0$). Meanwhile $f_{\vec{a}}'(r)=2^{j+1}a_jh'(2^{j+1}r-R)$ has absolute value equal to the length of an index-$k$ or $k+2$ geodesic from $Q$ to $x_1$, and hence is bounded above by $L$.  

By construction, $h''$ is positive precisely on the intervals $(\delta,R/4)$ and $(3R/4,R-\delta)$.  So since $a_j<0$ (so $f_{\vec{a}}''(r)$, which was earlier shown to be negative, has opposite sign to $h''(2^{j+1}r-R)$), we have $2^{j+1}r-R\in  (\delta,R/4)\cup (3R/4,R-\delta)$.   If $2^{j+1}r-R\in (\delta,R/4)$, then since  $a_j<0$, $h(\delta)=0$ and $h'$ is nonnegative and increasing on $[\delta,2^{j+1}r-R]$ we see that 
\[ f_{\vec{a}}(r)=a_jh(2^{j+1}r-R)=a_j\int_{\delta}^{2^{j+1}r-R}h'(s)ds \geq Ra_jh'(2^{j+1}r-R)\geq -RL \]  Similarly if $2^{j+1}r-R\in (3R/4,R-\delta)$ then since $h(R-\delta)=0$ and $h'$ is nonpositive and increasing on $[2^{j+1}r-R,R-\delta]$, \[ f_{\vec{a}}(r)=a_jh(2^{j+1}r-R)=-a_j\int_{2^{j+1}r-R}^{R-\delta}h'(s)ds \geq -Ra_jh'(2^{j+1}r-R)\geq -RL \]

Summing up, the relevant $\gamma_{(r,y)}$ have $\mathcal{A}_{H_{f_{\vec{a}}}}(\gamma_{(r,y)})=f_{\vec{a}}(r)-rf_{\vec{a}}'(r)$ where both $f_{\vec{a}}(r)\geq -RL$ and $rf_{\vec{a}}'(r)\geq -RL$, so part (i) of the Proposition holds with $C_0=2RL$.

The proof of part (ii) is a mirror image to that of part (i): one uses Proposition \ref{gradecompQ} to see that any such $(r,y)$ must have $f_{\vec{a}}''(r)>0$ and must correspond to a geodesic of Morse index $k$ or $k+2$.  This fact leads to an estimate $f_{\vec{a}}(r)\leq RL$, and hence to \[ \mathcal{A}_{H_{f_{\vec{a}}}}(\gamma_{(r,y)})=f_{\vec{a}}(r)-rf_{\vec{a}}'(r)\leq 2RL \]  Details are left to the reader.
\end{proof}




\begin{prop}\label{deep}  Under Assumption \ref{assm} there is a constant $A>0$ depending on $(N,g)$ and $h$  such that for all $\vec{a}\in \R^{\infty}$ obeying $\nu^*Q\pitchfork (\phi_{H_{f_{\vec{a}}}}^{1})^{-1}(T^{*}_{x_1}N)$:
\begin{itemize} \item[(i)] If $\min_i a_i< -A$ then there exists $(r,y)\in \mathcal{C}_{\frak{c}}(f_{\vec{a}},Q,x_1)$ such that $\mu(\gamma_{(r,y)})=d-k$ and $\mathcal{A}_{H_{f_{\vec{a}}}}(\gamma_{(r,y)})\leq \min_i a_i$.
\item[(ii)] If $\max_i a_i>A$ then there exists $(r,y)\in \mathcal{C}_{\frak{c}}(f_{\vec{a}},Q,x_1)$ such that $\mu(\gamma_{(r,y)})=n+k$ and $\mathcal{A}_{H_{f_{\vec{a}}}}(\gamma_{(r,y)})\geq \max_i a_i$.\end{itemize}
\end{prop}

\begin{proof}  Let $l_k$ denote the minimal length of an index-$k$ geodesic from $Q$ to $x_1$ in class $\frak{c}$ (such exists by Assumption \ref{assm}(i)).  Let $A=-\frac{l_k}{2h'(3R/4)}$ (note that it follows from the construction of $h$ that $h'(s)$ attains its minimal value, which is negative, at $s=3R/4$).  Thus if $b>2A$, then there is $s\in (R/2,3R/4)$ such that $bh'(s)=-l_k$.  

If $\min_i a_i<-A$, choose $j$ so that $a_j=\min_i a_i$.  For each $r\in [2^{-(j+1)}R,2^{-j}R]$ we have $f'_{\vec{a}}(r)=2^{j+1}a_j h'(2^{j+1}r-R)$.  So since $-2^{j+1}a_j>2A$ there is $r_0$ such that $2^{j+1}r_0-R\in (R/2,3R/4)$ and $f_{\vec{a}}'(r_0)=l_k$.  Since $a_j<0$ and, by the construction of $h$, $h''$ is negative on $(R/2,3R/4)$,  we have $f_{\vec{a}}''(r_0)>0$.  

Since $f_{\vec{a}}'(r_0)=l_k$ is the length of a Morse index-$k$ geodesic from $x_0$ to $x_1$ in the class $\frak{c}$, there is a corresponding element $(r_0,y)\in \mathcal{C}_{\frak{c}}(f_{\vec{a}},Q,x_1)$.  Since $f_{\vec{a}}'(r_0)$ and $f_{\vec{a}}''(r_0)$ are both 
positive we have $\mu(\gamma_{(r_0,y)})=d-k$ by Proposition \ref{gradecompQ}.  As for the action, let $r_1=\frac{3}{4}2^{-j}R$, so that $2^{j+1}r_1-R=\frac{R}{2}$ and so $f_{\vec{a}}(r_1)=a_j$, $f_{\vec{a}}'(r_1)=0$ and $f_{\vec{a}}''$ is positive on $[r_1,r_0]$.  Hence \begin{align*}
\mathcal{A}_{H_{f_{\vec{a}}}}(\gamma_{(r_0,y)})&=f_{\vec{a}}(r_0)-r_0f_{\vec{a}}'(r_0)
\\&=f_{\vec{a}}(r_1)+\int_{r_1}^{r_0}\frac{d}{dr}(f_{\vec{a}}(r)-rf'_{\vec{a}}(r))dr 
\\&=a_j-\int_{r_1}^{r_0}rf_{\vec{a}}''(r)dr\leq a_j=\min_i a_i \end{align*}
This proves (i)

The proof of (ii) is essentially the same: choose $j$ so that $a_j=\max_ia_i>A$, and then there is $r_0$ with $2^{j+1}r_0-R\in (R/2,3R/4)$ such that $f_{\vec{a}}'(r_0)=2^{j+1}a_jg'(2^{j+1}r_0-R)=-l_k$.  Moreover there is a corresponding $(r_0,y)\in \mathcal{C}_{\frak{c}}(f_{\vec{a}},Q,x_1)$, which will have $\mu(\gamma_{(r_0,y)})=n+k$ since $f_{\vec{a}}'(r_0),f_{\vec{a}}''(r_0)<0$, and will have $\mathcal{A}_{H_{f_{\vec{a}}}}(\gamma_{(r_0,y)})\geq a_j$ by a similar calculation to that above.
\end{proof}

This quickly leads to the key estimate of the boundary depth $B$ as defined in (\ref{Bdef}).

\begin{prop}\label{best} Under Assumption \ref{assm} there is a constant $C$ such that, for all $\vec{a}\in\R^{\infty}$, \begin{equation}\label{keyest} B\left(\nu^*Q,(\phi_{H_{f_{\vec{a}}}}^{1})^{-1}(T_{x_1}^{*}N)\right)\geq \|\vec{a}\|_{\infty}-C \end{equation}
\end{prop}

\begin{proof} By construction, each of the functions $f_{\vec{a}}'$ have only finitely many critical values.  From this (and the fact that $x_1$ is not a focal point of $Q$, so that there are only countably many geodesics from $Q$ to $x_1$) it is easy to see that, for any $\vec{b}\in \R^{\infty}$, all but countably many $\lambda\in \R$ have the property that the transversality condition (\ref{transcond}) holds for the function $H_{f_{\lambda\vec{b}}}=H_{\lambda f_{\vec{b}}}$.  So using the continuity property (\ref{betacont}) of the boundary depth it suffices to prove the proposition for those $\vec{a}\in \R^{\infty}$ such that $\nu^*Q\pitchfork (\phi_{H_{f_{\vec{a}}}}^{1})^{-1}(T^{*}_{x_1}N)$ (so that in particular the Floer complex $CF(\nu^*Q,T_{x_1}^{*}N;H_{f_{\vec{a}}})$ is well-defined).

 Let $C$ be the maximum of the constant $C_0$ from Proposition \ref{k-C} and the constant $A$ from Proposition \ref{deep}.  
Then (\ref{keyest}) is trivial if $\|\vec{a}\|_{\infty}\leq C$ (since $B$ is by definition always nonnegative), so we may assume that $\|\vec{a}\|_{\infty}> C$.

Now either $\|\vec{a}\|_{\infty}=-\min_i a_i$ or $\|\vec{a}\|_{\infty}=\max_i a_i$.  In the first case, we have $\min_i a_i<-C\leq -A$, So by Proposition \ref{deep}(i) there is a generator $\gamma_{(r,y)}$ for $CF_{\frak{c},d-k}(\nu^*Q,T_{x_1}^{*}N;H_{f_{\vec{a}}})$ with filtration level at most $-\|\vec{a}\|_{\infty}$.  But by Proposition \ref{k-C}(i), there are no generators of $CF_{\frak{c},d-k-1}(\nu^*Q,T_{x_1}^{*}N;H_{f_{\vec{a}}})$  having action less than or equal to $-C$.  This implies first that $\partial_{\mathbb{J}}\gamma_{(r,y)}=0$ (where $\partial_{\mathbb{J}}$ is the Floer differential), since $\partial_{\mathbb{J}}$ maps $CF_{\frak{c},d-k}$ to $CF_{\frak{c},d-k-1}$ and lowers filtration level, and the filtration level of $\gamma_{(r,y)}$ is at most $-\|\vec{a}\|_{\infty}< -C$.  But the homology of the chain complex $(CF_{\frak{c},d-k}(\nu^*Q,T_{x_1}^{*}N;H_{f_{\vec{a}}}),\partial_{\mathbb{J}})$ is zero (it is independent of the choice of compactly supported Hamiltonian $H$, and clearly vanishes when $H=0$ since $\nu^*Q\cap T_{x_1}^{*}N=\varnothing$).  So the fact that $\partial_{\mathbb{J}}\gamma_{(r,y)}=0$ implies that \[ \gamma_{(r,y)}\in Im\left(\partial_{\mathbb{J}}\co CF_{\frak{c},d-k+1}(\nu^*Q,T_{x_1}^{*}N;H_{f_{\vec{a}}})\to CF_{\frak{c},d-k}(\nu^*Q,T_{x_1}^{*}N;H_{f_{\vec{a}}})\right) \]  

By another application of Proposition \ref{k-C}(i), all nonzero elements of $CF_{\frak{c},d-k+1}(\nu^*Q,T_{x_1}^{*}N;H_{f_{\vec{a}}})$ have filtration level at least $-C$.  Thus $\gamma_{(r,y)}$ is an element of the image of the boundary operator having filtration level at most $-\|\vec{a}\|_{\infty}$, all of whose primitives have filtration level at least $-C$.  This proves (\ref{keyest}) in the case that $\|\vec{a}\|_{\infty}=-\min_i a_i$.

There remains the case that $\|\vec{a}\|_{\infty}=\max_i a_i$.  This case becomes essentially identical to the other one after we appeal to \cite[Corollary 1.4]{U10}, which shows that the boundary depth of the Floer complex $CF_{\frak{c}}(\nu^*Q,T_{x_1}^{*}N;H_{f_{\vec{a}}})$
is the same as that of its ``opposite complex,'' \emph{i.e.} of the filtered chain complex obtained from   $CF_{\frak{c}}(\nu^*Q,T_{x_1}^{*}N;H_{f_{\vec{a}}})$ by negating the grading, using the opposite filtration \[ \ell\left(\sum b_i\gamma_i\right)=\max\{-\mathcal{A}_{H_{f_{\vec{a}}}}(\gamma_i)|b_i\neq 0\}, \] and using as boundary operator the transpose of the original boundary operator $\partial_{\mathbb{J}}$.  In this opposite complex, the generator $\gamma_{(r,y)}$ provided by Proposition \ref{deep}(ii) will have grading $-n-k$ and filtration level at most $-\max_i a_i=-\|\vec{a}\|_{\infty}$, while Proposition \ref{k-C}(ii) shows that $\gamma_{(r,y)}$ is a cycle (and hence a boundary) in the opposite complex (since all index-$(-n-k-1)$ generators have higher filtration level), and that all primitives of $\gamma_{(r,y)}$ have filtration level at least $-C$.  Thus in the case that
$\|\vec{a}\|_{\infty}=\max_i a_i$ we again obtain (\ref{keyest}).
\end{proof}

The $k=0$ case of the following corollary  immediately implies Theorem \ref{mainlag}:

\begin{cor}\label{maincor}
Let the Riemannian manifold $(N,g)$, the submanifold $Q$, and the point $x_1\in N\setminus Q$ which is not a focal point of $Q$ obey Assumption \ref{assm}.  Let $U$ be a neighborhood of $0_N$, choose $R>0$ sufficiently small that the radius-$R$ disk bundle in $T^*N$ is contained in $U$, and construct the functions $f_{\vec{a}}$ as at the start of Section \ref{compute}.  Then the homomorphism $F\co \R^{\infty}\to C^{\infty}(T^*N)$ defined by $F(\vec{a})=H_{f_{\vec{a}}}$ satisfies the following properties, for all $\vec{a},\vec{b}\in \R^{\infty}$:\begin{itemize}
\item[(i)] $\phi_{F(\vec{a}+\vec{b})}^{1}=\phi_{F(\vec{a})}^{1}\circ \phi_{F(\vec{b})}^{1}$
\item[(ii)] For some constant $C$ independent of $\vec{a}$ and $\vec{b}$, \[ \delta\left(\phi_{F(\vec{a})}^{1}(T_{x_1}^{*}N),\phi_{F(\vec{b})}^{1}(T_{x_1}^{*}N)\right)\geq \|\vec{a}-\vec{b}\|_{\infty}-C \] and \[ 
\delta\left(\phi_{F(\vec{a})}^{1}(\nu^*Q),\phi_{F(\vec{b})}^{1}(\nu^*Q)\right)\geq \|\vec{a}-\vec{b}\|_{\infty}-C \] where $\delta$ denotes the Hofer distance on $\mathcal{L}(T_{x_1}^{*}N)$ or on $\mathcal{L}(\nu^*Q)$, respectively.\end{itemize}
\end{cor}

\begin{proof}
It is immediate from the definition of $f_{\vec{a}}$ that $f_{\vec{a}+\vec{b}}=f_{\vec{a}}+f_{\vec{b}}$; thus $F$ is indeed a homomorphism and the formula (\ref{hflow}) for the Hamiltonian flows of the functions $H_f$ directly implies (i).  As for (ii), note that for any Lagrangian submanifold $\Lambda$, the $Ham$-invariance of the Hofer distance $\delta$, together with (i), shows that we have \[ 
\delta(\phi_{F(\vec{a})}^{1}(\Lambda),\phi_{F(\vec{b})}^{1}(\Lambda))=\delta(\Lambda,(\phi_{F(\vec{a}-\vec{b})}^{1})^{-1}(\Lambda)).\]  

Now $B(\nu^*Q,T_{x_1}^{*}N)=0$ since the Floer complex $CF(\nu^*Q,T_{x_1}^{*}N;0)$ vanishes.  Hence by Corollaries \ref{lip} and \ref{best},
\[ 
\delta(T_{x_1}^{*}N,(\phi_{F(\vec{a}-\vec{b})}^{1})^{-1}(T_{x_1}^{*}N))\geq B(\nu^*Q,(\phi_{F(\vec{a}-\vec{b})}^{1})^{-1}(T_{x_1}^{*}N))\geq \|\vec{a}-\vec{b}\|_{\infty}-C,\] proving the first line of (ii).  As for the second line of (ii), note that for all $\vec{a}$ the Floer complex $CF(T_{x_1}^{*}Q,\nu^*Q;F(\vec{a}))$ is naturally isomorphic (modulo a shift of grading that results from our normalization convention) as a filtered chain complex to the opposite complex of   $CF(\nu^*Q,T_{x_1}^{*}N;F(-\vec{a}))$ (in the sense of \cite{U10} and of the last paragraph of Proposition \ref{best}).  So since by \cite[Corollary 1.4]{U10} the boundary depth is unchanged when we pass to the opposite complex, we obtain \begin{align*} \delta(\nu^*Q,(\phi_{F(\vec{a}-\vec{b})}^{1})^{-1}(\nu^*Q))&\geq \left|B(T_{x_1}^{*}N,(\phi_{F(\vec{a}-\vec{b})}^{1})^{-1}(\nu^*Q))-B(T_{x_1}^{*}N,\nu^*Q)\right|
\\& =B\left(T_{x_1}^{*}N,(\phi_{F(\vec{a}-\vec{b})}^{1})^{-1}(\nu^*Q)\right)\geq \|\vec{a}-\vec{b}\|_{\infty}-C.\end{align*}
\end{proof}

\end{document}